\documentclass[reqno]{amsart}

\usepackage{amssymb}
\usepackage{faktor}
\usepackage{bbm}
\usepackage[all,cmtip]{xy}
\usepackage{amsmath,amscd}
\usepackage{mathrsfs}
\usepackage{color}
\usepackage{mathtools}
\usepackage[mathcal]{euscript}
\usepackage{hyperref}
\hypersetup{colorlinks=true,linkcolor=blue}

\newcommand{\Hom}{\mathrm{Hom}}

\newcommand{\Ran}{\mathrm{Ran}}
\newcommand{\Conf}{\mathrm{Conf}}
\newcommand{\Diag}{\mathrm{Diag}}

\newcommand{\DDisk}{\mathcal{D}\mathrm{isk}}
\newcommand{\Sym}{\mathrm{Sym}}

\newcommand{\op}{{op}}

\newcommand{\Shv}{\mathrm{Shv}}

\newcommand{\Emb}{\mathrm{Emb}}
\newcommand{\MMfld}{\mathcal{M}\mathrm{fld}}
\newcommand{\EEuc}{\mathcal{E}\mathrm{uc}}
\newcommand{\Mfld}{\mathrm{Mfld}}

\newcommand{\Top}{\mathrm{Top}}

\newcommand{\Alg}{\mathrm{Alg}}

\newcommand{\Mod}{\mathrm{Mod}}

\newcommand{\Coalg}{\mathrm{Coalg}}

\newcommand{\Lie}{\mathcal{L}}
\newcommand{\Fr}{\mathrm{Fr}}

\newcommand{\Ind}{\mathrm{Ind}}
\newcommand{\Fun}{\mathrm{Fun}}
\newcommand{\Fact}{\mathrm{Fact}}
\newcommand{\Euc}{\mathrm{Euc}}

\newcommand{\Surj}{\mathrm{Surj}}
\newcommand{\ch}{\mathrm{ch}}

\newcommand{\gr}{\mathrm{gr}}
\newcommand{\loc}{\mathrm{loc}}

\newcommand{\id}{\mathrm{id}}
\newcommand{\cbl}{\mathrm{cbl}}

\newcommand{\cShv}{\mathrm{cShv}}

\newcommand{\D}{\mathcal{D}}
\newcommand{\W}{\mathcal{W}}

\newcommand{\C}{\mathcal{C}}

\newcommand{\Map}{\mathrm{Map}}
\newcommand{\V}{\mathcal{V}}
\DeclareMathOperator*{\im}{\mathrm{im}}

\newcommand{\Disk}{\mathrm{Disk}}
\newcommand{\Cov}{\mathrm{Cov}}
\newcommand{\Disj}{\mathrm{Disj}}
\newcommand{\pt}{\mathrm{pt}}
\newcommand{\lax}{\mathrm{lax}}
\newcommand{\oplax}{\mathrm{oplax}}

\DeclareMathOperator*{\colim}{\mathrm{colim}}
\DeclareMathOperator*{\Fin}{\mathrm{Fin}}

\newcommand{\Com}{\mathrm{Com}}

\theoremstyle{definition}\newtheorem{definition}{Definition}[subsection]
\theoremstyle{remark}\newtheorem{remark}[definition]{Remark}
\theoremstyle{definition}
\theoremstyle{definition}
\theoremstyle{definition}\newtheorem{example}[definition]{Example}
\theoremstyle{theorem}\newtheorem{proposition}[definition]{Proposition}
\theoremstyle{theorem}\newtheorem{lemma}[definition]{Lemma}
\theoremstyle{definition}
\theoremstyle{theorem}\newtheorem{corollary}[definition]{Corollary}
\theoremstyle{theorem}\newtheorem{theorem}[definition]{Theorem}
\theoremstyle{definition}
\theoremstyle{remark}
\theoremstyle{theorem}\newtheorem{bigtheorem}{Theorem}
\theoremstyle{theorem}\newtheorem{introtheorem}{Theorem}

\makeatletter
\@addtoreset{definition}{section}
\makeatother

\title{Higher enveloping algebras}
\author{Ben Knudsen}
\date{} 

\begin{document}

\begin{abstract}
We provide spectral Lie algebras with enveloping algebras over the operad of little $G$-framed $n$-dimensional disks for any choice of dimension $n$ and structure group $G$, and we describe these objects in two complementary ways. The first description is an abstract characterization by a universal mapping property, which witnesses the higher enveloping algebra as the value of a left adjoint in an adjunction. The second, a generalization of the Poincar\'{e}-Birkhoff-Witt theorem, provides a concrete formula in terms of Lie algebra homology. Our construction pairs the theories of Koszul duality and Day convolution in order to lift to the world of higher algebra the fundamental combinatorics of Beilinson-Drinfeld's theory of chiral algebras. Like that theory, ours is intimately linked to the geometry of configuration spaces and has the study of these spaces among its applications. We use it here to show that the stable homotopy types of configuration spaces are proper homotopy invariants.
\end{abstract}

\maketitle

\section{Introduction} As the structure inherited by the tangent space to the identity element of a Lie group, Lie algebras are classically tied to the smooth geometry of manifolds. In this work, we explore a more primitive source for this same type of algebraic structure. Our guiding philosophy is that the Lie bracket is an emergent feature of the \emph{topology} of manifolds. 

Our eventual goal is to formulate and prove a statement about the relationship between Lie algebras and manifolds in terms of algebras over the operad of little $n$-dimensional disks. Before doing so, we undertake a brief tour of some of the manifestations of this relationship, beginning with its first appearance in the study of configuration spaces.

\subsection{Configuration spaces} In his study of the braid groups \cite{Arnold:CRCBG}, Arnold was led to consider the cohomology of $\Conf_k(\mathbb{R}^2)$, where for a manifold $M$ we write \[\Conf_k(M)=\{(x_1,\ldots, x_k)\in M^k: x_i\neq x_j\text{ if } i\neq j\}\] for the configuration space of $k$ ordered points in $M$. His approach, later adapted to Euclidean spaces of higher dimension by Cohen \cite[III]{CohenLadaMay:HILS}, was to exploit a natural family of cohomology classes $\{\omega_{ij}\}_{i\neq j}$, where $\omega_{ij}$ is dual to the submanifold $M_{ij}\subseteq\Conf_k(\mathbb{R}^n)$ defined by allowing the points $x_i$ and $x_j$ to orbit freely at fixed distance about their center of mass (see \cite{Sinha:HLDO} for a beautiful discussion of this point of view). More precisely, $\omega_{ij}$ is the class pulled back under the Gauss map \begin{align*}\Conf_k(\mathbb{R}^n)&\to S^{n-1}\\
(x_1,\ldots, x_k)&\mapsto\displaystyle\frac{x_i-x_j}{\|x_i-x_j\|}\end{align*} from the standard volume form on $S^{n-1}$. The cohomology ring of $\Conf_k(\mathbb{R}^n)$ is generated by these classes subject only to the following two relations, which we have named suggestively: \begin{align*}\omega_{ij}=(-1)^{n-1}\omega_{ji}&\qquad(\text{antisymmetry})\\
\omega_{ij}\omega_{jk}+\omega_{ki}\omega_{ij}+\omega_{jk}\omega_{ki}=0&\qquad (\text{Jacobi})\end{align*}To explain in what sense these names are deserved, we turn to the theory of iterated loop spaces.

\subsection{Loop spaces} Let $(X,x_0)$ be a pointed space and $\Omega^nX:=\Map\big((I^n,\partial I^n), (X,x_0)\big)$ the associated $n$-fold loop space, where $I=[-1,1]$. Given rectilinear self-embeddings $\{f_1,\ldots, f_k\}$ of $I^n$ with disjoint images, we obtain a $k$-to-one operation $m_{f_1,\ldots, f_k}$ on $\Omega^nX$ by setting $$m_{f_1,\ldots, f_k}\big(\varphi_1,\ldots, \varphi_k\big)(t)=\begin{cases}
\varphi_i(f_i^{-1}(t))\qquad& t\in\im f_i\\
x_0\qquad & \text{else.}
\end{cases}$$ As the embeddings vary, we obtain a map $$\xymatrix{m_k:E_n(k)\times \big(\Omega^nX\big)^k\ar[r]&\Omega^n X,}$$ where $E_n(k)$ is the space of rectilinear embeddings of $\amalg_k I^n$ into $I^n$. The compatibility relations among the various $m_k$ are summarized by saying that the collection $E_n=\{E_n(k)\}_{k\geq0}$ forms an \emph{operad} and that the maps $m_k$ endow $\Omega^nX$ with the structure of an $E_n$-\emph{algebra}---see \cite{BoardmanVogt:HIASTS} and \cite{May:GILS} for early articulations of these ideas.

The connection to Lie algebras lies in the observation that the map $
E_n(k)\to\Conf_k(\mathbb{R}^n)$ given by evaluation at the origin is a homotopy equivalence for each $k$, so that, at the level of homology, there is an induced operation $$\xymatrix{[M_{12}]\otimes H_*(\Omega^nX)^{\otimes 2}\ar[r]&H_*(\Omega^nX),}$$ called the \emph{Browder bracket} (see \cite{Browder:HOLS}). Equipped with this operation, the shifted homology $H_{*+n-1}(\Omega^nX)$ obtains the structure of a graded Lie algebra, for which antisymmetry and the Jacobi identity are guaranteed by the corresponding relations in the cohomology of configuration spaces.

\subsection{Enveloping algebras} More homotopically, one might expect that the $n-1$-fold suspension of a differential graded $E_n$-algebra---such as the singular chain complex $C_*(\Omega^nX)$---should naturally carry the structure of a Lie algebra, perhaps up to homotopy. 

Since $E_1$ is equivalent to the operad governing associative algebras, the statement for $n=1$ is nothing more than the familiar fact that the commutator in an associative algebra is a Lie bracket. Classically, this observation is the beginning of a fruitful interplay between these two types of algebra, the avatar of which is the \emph{universal enveloping algebra}, the left adjoint to the forgetful functor taking an associative algebra to its commutator Lie algebra.

Our first result generalizes this situation to higher dimensions and nonzero characteristics. We will actually prove a more general version of this result, stated below in \S\ref{sec:structured manifolds}, which takes into account the action of a structure group $G\to O(n).$

\begin{introtheorem}[$G=\{e\}$]\label{thm:baby adjunction}
Let $\C$ be a stable, presentably symmetric monoidal $\infty$-category. There is an adjunction of $\infty$-categories between nonunital $\mathbb{E}_n$-algebras and Lie algebras in $\C$ fitting into a commuting diagram of adjunctions \[\xymatrix{
\Alg_{\Lie}(\C)\ar@{-->}@/^1.2pc/[rr]^{U_n}\ar[dd]&&\Alg^\mathrm{nu}_{\mathbb{E}_n}(\C)\ar[dd]\ar@{-->}[ll]\\\\
\C\ar@/_1.2pc/[rr]_{[1-n]}\ar@/^1.2pc/[uu]^-{\Lie}&&\C\ar@/_1.2pc/[uu]_-{\mathbb{E}^\mathrm{nu}_n}\ar[ll]_-{
[n-1]}
}\] 
\end{introtheorem} 

Here we use the notation $[r]$ of homological algebra for the $r$-fold suspension in $\C$, and $\Alg_\Lie(\C)$ denotes the $\infty$-category of \emph{spectral Lie algebras} in $\C$; see \S\ref{sec:Lie section}. 

The left adjoint $U_n$ is the titular higher enveloping algebra functor, and its construction is the main objective of this paper; however, almost all of the work will go toward exhibiting the right adjoint forgetful functor---see \S\ref{section:sketch} below for an outline of this argument. The functor $U_n$ provides a wide class of examples of $\mathbb{E}_n$-algebras, one which crucially includes the free algebras; indeed, a free $\mathbb{E}_n$-algebra---which is simply a tensor algebra when $n=1$---is the higher enveloping algebra of a free Lie algebra. Since free $\mathbb{E}_n$-algebras are built from the homotopy types of configuration spaces, this formal consequence of Theorem \ref{thm:baby adjunction} suggests an approach to the study of these spaces premised on Lie algebras. We will return to this idea in \S\ref{section:PBW} below.

We close this section by pointing out that, if $\C$ is taken to be the underlying $\infty$-category of the model category of chain complexes over a field $k$ of characteristic zero, then $\Alg_\Lie(\C)$ is equivalent to the underlying $\infty$-category of the category of differential graded Lie algebras with the induced model structure. Thus, Theorem \ref{thm:baby adjunction} implies that dg $E_n$-algebras---$C_*(\Omega^nX;\mathbb{Q})$, for example---carry a shifted dg Lie structure, up to homotopy; in particular, there is an induced adjunction at the level of homotopy categories. A Quillen adjunction, given by induction and restriction along the map of operads constructed in \cite{Fresse:KDCCILSS}, is also available in this case. The value of the derived left adjoint of this adjunction is another candidate for a higher enveloping algebra, and we expect the two to coincide. We choose to offer this alternate construction both because of its greater generality and because it admits a remarkably explicit description, to which we now turn.

\subsection{Poincar\'{e}-Birkhoff-Witt}\label{section:PBW} Classically, one forces the universal enveloping algebra $U(L)$ to have the desired universal mapping property by defining it as the quotient of the free associative algebra on $L$ by the relation $$x\otimes y-y\otimes x=[x,y].$$ This relation is inhomogeneous with the respect to the natural grading of the tensor algebra by word length and, after passing to the associated graded for the induced filtration of $U(L)$, becomes the defining relation of the symmetric algebra. The classical Poincar\'{e}-Birkhoff-Witt theorem formalizes this observation, asserting an isomorphism \[\mathrm{gr}\,U(L)\cong \Sym(L).\] In particular, at the level of underlying vector spaces, the universal enveloping algebra is indistinguishable from the symmetric algebra.

We prove the following generalization of this result (see \S\ref{sec:enveloping algebras} for the full version).

\begin{introtheorem}[$G=\{e\}$]\label{thm:baby PBW}
Let $L$ be a Lie algebra in $\C$. There is a natural equivalence of augmented $\mathbb{E}_n$-algebras \[U_n(L)\oplus 1_\C\simeq C^\Lie(\Omega^n L).\]
\end{introtheorem} Here, $C^\Lie$ denotes the functor of Lie chains, defined in this generality as the monadic bar construction against the free Lie algebra monad.

In order to understand how this result generalizes the Poincar\'{e}-Birkhoff-Witt theorem, we recall that, in characteristic zero, the functor $C^\Lie$ is modeled by the classical Chevalley-Eilenberg complex $CE(L)$, which is the graded vector space $\Sym(L[1])$ equipped with a differential determined by the Lie bracket of $L$ (see \cite[6]{Fresse:KDOHPP} for a comparison), while the cotensor $\Omega^nL\simeq L^{(\mathbb{R}^n)^+}$ is modeled by the tensor product $A_{PL}(S^n)\otimes L$, where $A_{PL}$ is the functor of (reduced) piecewise-linear de Rham forms \cite[4.8.3]{Hinich:HAHA}. Since the Lie bracket in this tensor product is homotopically trivial, we have the equivalence of underlying chain complexes \[U_n(L)\simeq CE(A_{PL}(S^n)\otimes L)\simeq \Sym(L[1-n]).\]

The algebra appearing on the righthand side of the equivalence of Theorem \ref{thm:PBW} has long been an object of interest. In particular, in its guise as a factorization algebra, it is an important player in the approach to quantum field theory pioneered by Costello and Gwilliam---see \cite[4.6-7]{Gwilliam:FAFFT} and \cite[3.6.6]{CostelloGwilliam:FAQFTI}, for example. From this point of view, Theorem \ref{thm:baby adjunction} may be interpreted as endowing this well-known object with a useful universal property.

In combination with the theory of factorization homology, Theorem \ref{thm:PBW} offers excellent computational opportunities. Some of these are explored in \cite{Knudsen:BNSCSVFH}, \cite{DrummondColeKnudsen:BNCSS}, and \cite{BrantnerHahnKnudsen:HLALTTCS}, which employ the theory developed here in studying the homology of configuration spaces of manifolds. As a further illustration of its effectiveness, we include the following application.

\begin{introtheorem}\label{thm:configurations}
Let $M$ be an $n$-manifold. For any $k\geq0$, the $\Sigma_k$-equivariant homotopy type of $\Sigma_+^\infty\Conf_k(M)$ depends only on the homotopy type of the one-point compactification $M^+$. In particular, the stable homotopy types of the unordered configuration spaces of $M$ depend only on the homotopy type of $M^+$.
\end{introtheorem}

This result improves on \cite[B]{AouinaKlein:HICS}, which is non-equivariant and requires $M$ to be compact, connected, and piecewise-linear; however, that work also provides explicit bounds on the number of suspensions necessary to achieve homotopy invariance. Our methods are unable to replicate such bounds.

\subsection{The Ran space} Our approach to the results outlined above lies in importing the ideas of \cite{BeilinsonDrinfeld:CA} and \cite{FrancisGaitsgory:CKD} from algebraic geometry into topology. We now recall some of the context of those works. Throughout this motivational section, the reader unfamiliar with the theory of $D$-modules is invited instead to imagine sheaves, as $D$-modules will play no role in the remainder of the paper.

The starting point is the so-called Ran space of a variety $X$, which may be thought of heuristically as the space \[\mathrm{Ran}(X)=\{S\subset X: 0<|S|<\infty\}\] of finite subsets of $X$ (in the topological context, this heuristic is a definition). The collection of $D$-modules on $\Ran(X)$ carries a remarkable monoidal structure, the \emph{chiral tensor product}, which is computed on stalks by the formula \[(F\otimes^\ch G)_{S}=\bigoplus_{S=S_1\amalg S_2}F_{S_1}\otimes G_{S_2}.\] Motivated by conformal field theory, Beilinson and Drinfeld were led to consider a certain algebraic structure emerging from this tensor product, which may be phrased either in cocommutative terms as a \emph{factorization algebra} or in Lie terms as a \emph{chiral algebra}. As shown in \cite{FrancisGaitsgory:CKD}, this duality is an instance of the Koszul duality between cocommutative coalgebras and Lie algebras.

One advantage of working on the chiral or Lie side is that one is able to construct examples by passing through a second monoidal structure, the \emph{star tensor product}, which is computed on stalks as \[(F\otimes^\star G)_{S}=\bigoplus_{S=S_1\cup S_2}F_{S_1}\otimes G_{S_2}.\] This tensor product interpolates between the ordinary tensor product and the chiral tensor product, and through it an ordinary Lie algebra gives rise to a chiral Lie algebra, its \emph{chiral envelope}, and thereby to a factorization algebra.

The connection to our previous discussion is indicated by a theorem of Lurie \cite[5.5.4.10]{Lurie:HA}, which asserts that $\mathbb{E}_n$-algebras may be realized as certain ``factorizable'' cosheaves on the Ran space of $\mathbb{R}^n$---roughly, these are cosheaves $F$ equipped with equivalences $F_S\simeq \bigotimes_{x\in S} F_x$ for every finite subset $S\subseteq M$. Motivated by this theorem, our strategy will be to build a topological framework corresponding to that of \cite{BeilinsonDrinfeld:CA} and \cite{FrancisGaitsgory:CKD}, and to construct higher enveloping algebras by ``passing to the chiral side.''

In more detail, we introduce the \emph{absolute Ran category}, a category fibered over the category of framed manifolds, and an $\infty$-category of \emph{constructible cosheaves} on this category. We define two symmetric monoidal structures on this $\infty$-category and proceed according to the following table of analogies:\\

\begin{center} \begin{tabular}{ | c | c | } \hline $D$-module on $\Ran(X)$ & constructible cosheaf on $\Ran_n$ \\ 
\hline chiral tensor product & disjoint tensor product \\ 
\hline star tensor product & overlapping tensor product \\ 
\hline factorization algebra on $X$ & nonunital $\mathbb{E}_n$-algebra \\ 
\hline chiral envelope & higher enveloping algebra\\
\hline\end{tabular} \end{center}\vspace{.1in}

As mentioned above, factorization algebras in the context of \cite{BeilinsonDrinfeld:CA} and \cite{FrancisGaitsgory:CKD} are defined to be a special kind of cocommutative coalgebra, while $\mathbb{E}_n$-algebras are modeled rather as certain symmetric monoidal functors. Thus, if this table of analogies is to have any sense, we must answer the following question: \emph{how is a symmetric monoidal functor like a cocommutative coalgebra?}

\subsection{Day convolution} In a very general setting, when $\W$ is symmetric monoidal, $\Fun(\V,\W)$ carries the rather pedestrian symmetric monoidal structure of pointwise tensor product. If $\V$ is also symmetric monoidal, however, one obtains a more interesting symmetric monoidal structure, that of \emph{Day convolution}. Defined in the setting of ordinary categories in \cite{Day:CCF} and in the $\infty$-categorical context \cite{Glasman:DCIC} (see also \cite[4.8.1.12]{Lurie:HA}), the convolution of functors $F$ and $G$ is the left Kan extension in the following diagram: \[\xymatrix{\V\times\V\ar[d]_-{\otimes_\V}\ar[r]^-{F\times G}& \W\times\W\ar[r]^-{\otimes_\W}&\W\\
\V
}\] One of the attractive features of this tensor product is that providing a functor $F$ with the structure of a commutative algebra for Day convolution is equivalent to providing $F$ with a lax monoidal structure, with the lax structure maps supplied by the components \[\xymatrix{
F(v_1)\otimes F(v_2)\ar[r]&\displaystyle\colim_{v_1\otimes v_1\to v}F(v_1)\otimes F(v_2)\ar[r]&F(v).
}\] of the algebra structure map. In this way, symmetric monoidal functors from $\V$ to $\W$ naturally form a full subcategory of commutative algebras for Day convolution. 

In order to answer the question posed above, we turn the discussion of the previous paragraph on its head. Taking the right Kan extension instead of the left produces a different tensor product, for which a cocommutative coalgebra is precisely an \emph{oplax} functor. Since a symmetric monoidal functor may be viewed equally as a lax functor with an extra property or as an oplax functor with an extra property, we find that such functors naturally form a full subcategory of cocommutative coalgebras for this alternate form of convolution.

\subsection{Sketch of argument}\label{section:sketch} Once fleshed out, the ideas alluded to so far constitute a passage among four different models for $\mathbb{E}_n$-algebras, as indicated in the following schematic: \[\left\{
\begin{tabular}{c}
    topological \\
    model
\end{tabular}
\right\}\overset{(\ref{prop:discrete model})}{\simeq}\left\{
\begin{tabular}{c}
    discrete \\
    model
\end{tabular}
\right\}\overset{(\ref{prop:factorizable 2})}{\simeq}\left\{
\begin{tabular}{c}
    cocommutative \\
    model
\end{tabular}
\right\}\overset{(\ref{lem:factorization})}{\simeq}\left\{
\begin{tabular}{c}
    Lie \\
    model
\end{tabular}
\right\}.\] The leftmost equivalence, which closely follows the ideas of \cite[5]{Lurie:HA}, models $\mathbb{E}_n$-algebras as symmetric monoidal functors out of a certain category disjoint unions of disks, while the middle equivalence is provided by the theory of Day convolution. The rightmost equivalence passes through Koszul duality between Lie algebras and cocommutative coalgebras as in \cite{Quillen:RHT}, the absence of connectivity hypotheses in this equivalence owing to the notion of symmetric monoidal \emph{pro-nilpotence} introduced in \cite{FrancisGaitsgory:CKD}. Combining these equivalences produces an embedding of nonunital $\mathbb{E}_n$-algebras into the $\infty$-category of Lie algebras in the disjoint monoidal structure on constructible cosheaves on the absolute Ran category. 

Now, the disjoint and overlapping tensor products are related by a natural transformation, so that Lie algebras for the former may be viewed as Lie algebras for the latter. We use this natural relationship to produce the forgetful functor of Theorem \ref{thm:adjunction} by embedding its source and target into these two categories of Lie algebras, as depicted in the following diagram:

 \[\xymatrix{
\Alg^\mathrm{nu}_{\mathbb{E}_n}(\C)\ar@{-->}[d]\ar@{^{(}->}[r]^-{(\ref{cor:lie model})}&\Alg_\Lie(\cShv^\cbl(\Ran_n,\C)_\amalg)\ar[d]\\
\Alg_{\Lie}(\C)\ar@{^{(}->}[r]^-{(\ref{prop:overlapping lie})}&\Alg_\Lie(\cShv^\cbl(\Ran_n,\C)_\cup).
}\] 

From here, Theorem \ref{thm:PBW} is within close reach, since the functor $C^\Lie$, as the avatar of Koszul duality, is baked into the definition of the forgetful functor.

\subsection{Linear outline}

\begin{enumerate}
\item[\S\ref{sec:structured manifolds}] We discuss $G$-framed manifolds and related algebraic structures, comparing topological and discrete incarnations.\\
\item[\S\ref{sec:Ran section}-\S\ref{sec:push pull section}] We introduce the absolute Ran category and constructible cosheaves on it, and we discuss various push and pull operations for the latter.\\
\item[\S\ref{sec:factorizable coalgebras}-\S\ref{sec:nilpotence}] We introduce the disjoint monoidal structure on constructible cosheaves, realizing disk algebras as \emph{factorizable} cocommutative coalgebras therein, and we show that this monoidal structure is pro-nilpotent.\\
\item[\S\ref{sec:Lie section}-\S\ref{sec:duality and factorization}] We discuss the Koszul duality relating cocommutative coalgebras and Lie algebras, and we characterize those Lie algebras whose associated cocommutative coalgebras are factorizable, thereby obtaining a Lie theoretic model for disk algebras.\\
\item[\S\ref{sec:enveloping algebras}-\S\ref{sec:configurations section}] We complete the proofs of the main theorems.\\
\item[\S\ref{sec:functoriality section}-\S\ref{sec:monoidal section}] We provide various deferred constructions and arguments.
\end{enumerate}

\subsection{Acknowledgments} This paper is largely an attempt at drawing the straightest lines possible among \cite[5]{Lurie:HA}, \cite{FrancisGaitsgory:CKD}, \cite{AyalaFrancis:FHTM}, and \cite{Glasman:DCIC} and it is a pleasure to acknowledge my intellectual debt to the authors of these works. Of equal importance are the debts that lie outside the scope of the literature, since the results proven here, as well as the broad strokes of their proofs, have been ``in the air'' for some time and are certainly well known to experts; in particular, statements of versions of these results appeared in preliminary versions of \cite{AyalaFrancis:FHTM}. I am grateful for the intellectual generosity of David Ayala and John Francis, without which this paper would not exist. 

I would like to thank Owen Gwilliam, Haynes Miller, and Emily Riehl for helpful conversations. I thank the referee for her or his comments and for the pun ``$E_n$veloping algebra.'' Finally, I would like to express my gratitude to the Mathematisches Forschungsinstitut Oberwolfach, the Hausdorff Research Institute for Mathematics in Bonn, and the University of Copenhagen for their hospitality during the writing of this paper. 

This paper is derived from my Ph.D. thesis \cite{Knudsen:HEACSM}, and part of its writing was carried out with the support of an NSF postdoctoral fellowship.

\subsection{Conventions}
\begin{enumerate}
\item We work in an $\infty$-categorical context, where for us an $\infty$-category will always mean a quasicategory. The standard references here are \cite{Lurie:HTT} and \cite{Lurie:HA}. We use the same symbol for an ordinary or topological category and the corresponding $\infty$-category, as well as for a topological space and its $\infty$-groupoid of singular simplices. The terms initial and final are always understood in the $\infty$-categorical sense, and (co)limits are alway $\infty$-categorical (co)limits, which correspond (in the presence of a comparison to some model category) to homotopy (co)limits. 
\item In a monoidal context, we write $\Sym^k(c):=(c^{\otimes k})_{\Sigma_k}$, where, in accordance with the previous convention, we intend the $\infty$-categorical coinvariants. We may distinguish among multiple monoidal structures on the same underlying $\infty$-category using superscripts (e.g., $\otimes^\amalg$ and $\otimes^\cup$). When performing constructions involving a monoidal structure, we may use the corresponding symbol to indicate which monoidal structure is intended (e.g., $\Sym^k_\amalg$).
\item We use the superscript ``$\mathrm{nu}$'' to indicate nonunital algebraic structures and non-counital coalgebraic structures (e.g. $\Coalg_\Com^{\mathrm{nu}}(\V)$ for the $\infty$-category of non-counital cocommutative coalgebras in the possibly nonunital symmetric monoidal $\infty$-category $\V$). For more on nonunitality in $\infty$-categorical algebra, we refer the reader to \cite[5.4]{Lurie:HA}.
\item Where possible, left adjoints are written on the left to aid in their identification. In more spatially complicated situations, left adjoints are written with bent arrows. A hooked arrow indicates a fully faithful functor.
\item In an abstract setting, restriction along the functor $j$ is denoted $j^\natural$ (we owe this piece of notation to David Gepner), left Kan extension by $j_!$, and right Kan estension by $j_*$. In situations more closely related to geometry, we may employ the notation $j^!$ and $j^*$, where $j^!$ is always the right adjoint of $j_!$ and $j^*$ the left adjoint of $j_*$. The reader is warned that the functors $j^*$ and $j^!$ may have no relation to the functor $j^\natural$.
\item We write $c[k]$ for the $k$-fold suspension of the object $c$. With a structure group $G\to \Top(n)$ in mind, we reserve the notation $\Sigma^nc=(\mathbb{R}^n)^+\otimes c$ for the same object with $G$-action inherited from $\mathbb{R}^n$; if $c$ already carried a $G$-action, then $\Sigma^nc$ carries the diagonal action. Thus, we have the following commuting diagram: \[\xymatrix{
\Mod_G(\C)\ar[rr]^-{\Sigma^n}&&\Mod_G(\C)\ar[dd]\\\\
\C\ar[uurr]^-{\Sigma^n}\ar[uu]^-{\mathrm{triv}}\ar[rr]^{[n]}&&\C
}\] The dual remarks apply to $\Omega^n$. We regard $\Mod_G$ as symmetric monoidal with the Cartesian monoidal structure (or equally, by stability, the coCartesian monoidal structure).
\item Every manifold considered herein may be embedded as the interior of a compact manifold with boundary (such an embedding is not part of the data).
\item We write $\Fin$ for the category of finite sets and $\Surj$ for the wide subcategory on the surjective functions. An object of $\Fin$ is typically denoted by a letter such as $I$ or $J$, and we write $U_I$ for the $I$-indexed set $\{U_i\}_{i\in I}$. For a topological space $X$ with connected components $\{X_i\}_{i\in I}$, we identify $X$ with the $I$-indexed set $X_I$.
\end{enumerate}

\section{The Ran space}\subsection{Structured manifolds and disk algebras}\label{sec:structured manifolds}

Following \cite[5.1]{Lurie:HA}, we write $\mathbb{E}_n$ for the $\infty$-operad obtained from the topological operad of rectilinear embeddings of open $n$-dimensional unit cubes (see \cite{BoardmanVogt:HIASTS}, \cite{May:GILS}). The $\infty$-operad $\mathbb{E}_n$ and its algebras are the fundamental objects of study for us, but it will be useful to work in a slightly more general setting. The references for the material in this section are \cite{AyalaFrancis:FHTM} and \cite[5]{Lurie:HA}, although, for the sake of typographical clarity, we shall at times depart from the notation of these references. We make almost no use of the theory of $\infty$-operads outside of this section, but the reader seeking further information should consult \cite{Lurie:HA}.

As usual, we let $\Top(n)$ denote the topological group of self-homeomorphisms of $\mathbb{R}^n$. Recall that an $n$-manifold has a tangent microbundle classified by a map $M\to B\Top(n)$, well-defined up to homotopy, called the \emph{tangent classifier} of $M$. 

\begin{definition} Let $G\to \Top(n)$ be a continuous group homomorphism. A $G$-\emph{framing} on an $n$-manifold $M$ is a map $M\to BG$ lifting the homotopy class of the tangent classifier.
\end{definition}

Following \cite{AyalaFrancis:FHTM}, there is a topological category $\MMfld_G$ of $G$-framed $n$-manifolds and $G$-framed embeddings among them, where the space of $G$-framed embeddings from $M$ to $N$ is defined as the homotopy pullback \[\xymatrix{\Emb^G(M,N)\ar[d]\ar[r]&\Map_{/BG}(M,N)\ar[d]\\
\Emb(M,N)\ar[r]&\Map_{/B\Top(n)}(M,N).
}
\] Disjoint union of $G$-framed manifolds endows $\MMfld_G$ with a symmetric monoidal structure. 

\begin{remark}
In \cite{AyalaFrancis:FHTM}, the authors consider manifolds structured by an arbitrary fibration $B\to B\Top(n)$. In the notation of that reference, $\MMfld_G:=\MMfld_n^{BG}$. Our results, and their proofs, are valid in this more general setting.
\end{remark}

\begin{definition}
The $\infty$-operad $\mathbb{E}_G$ is the operadic nerve of the endomorphism operad of $\mathbb{R}^n$ in $\MMfld_G$. 
\end{definition}

\begin{remark}
In the notation of \cite[5]{Lurie:HA}, $\mathbb{E}_G:=\mathbb{E}_{BG}$. This object is an $\infty$-operadic analogue of the semidirect product of $G$ with the little disks operad in the sense of \cite{SalvatoreWahl:FDOBVA}.
\end{remark}

\begin{example}
Taking $G$ to be the trivial group, \cite[5.4.2.15]{Lurie:HA} provides an equivalence $\mathbb{E}_G\simeq \mathbb{E}_n$ of $\infty$-operads. 
\end{example}

In general, an $\mathbb{E}_G$-algebra may be thought of informally an $\mathbb{E}_n$-algebra together with a compatible $G$-action. Indeed, there is a forgetful functor \[\Alg_{\mathbb{E}_G}(\C)\to \Mod_G(\C):=\Fun(BG, \C),\] which is obtained by restricting the action of $\mathbb{E}_G$ to the space of unary operations, the topological monoid $\Emb^G(\mathbb{R}^n,\mathbb{R}^n)\simeq G$ (see \cite[2.8]{AyalaFrancis:FHTM}). 

In what follows, we will be interested in the $\infty$-category $\Alg_{\mathbb{E}_G}^\mathrm{nu}(\C)$ of nonunital $\mathbb{E}_G$-algebras in a suitable stable target $\C$. We now state the full version of the main theorem.

\begin{bigtheorem}\label{thm:adjunction}
Let $\C$ be a stable, presentably symmetric monoidal $\infty$-category and $G\to \Top(n)$ a continuous group homomorphism. There is a forgetful functor from nonunital $\mathbb{E}_G$-algebras to Lie algebras in $G$-modules in $\C$ fitting into a commuting diagram of adjunctions \[\xymatrix{
\Alg_{\Lie}(\Mod_G(\C))\ar@/^1.2pc/[rr]^{U_G}\ar[dd]&&\Alg^\mathrm{nu}_{\mathbb{E}_G}(\C)\ar[dd]\ar[ll]\\\\
\Mod_G(\C)\ar@/_1.2pc/[rr]_{\Omega^n[1]}\ar@/^1.2pc/[uu]^-{\Lie}&&\Mod_G(\C)\ar@/_1.2pc/[uu]_-{\mathbb{E}^\mathrm{nu}_G}\ar[ll]_-{
\Sigma^n[-1]}
}\] 
\end{bigtheorem} 

The proof will require a significant amount of preparatory work and is completed below in \S\ref{sec:enveloping algebras}. The first reduction is to replace $\mathbb{E}_G$ with a discrete model. 

\begin{definition}
We define three categories. 
\begin{enumerate}
\item $\Mfld_G$ is the ordinary category with the same objects and morphisms as $\MMfld_G$.
\item $\D_G\subseteq \Mfld_G$ is the subcategory with objects the $G$-framed manifolds homeomorphic to $\amalg_k\mathbb{R}^n$ for some $k\geq0$ and morphisms the $G$-framed embeddings that surject on $\pi_0$.
\item $\Euc_G\subseteq \Mfld_G$ is the full subcategory spanned by the $G$-framed manifolds homeomorphic to $\mathbb{R}^n$.
\end{enumerate}
\end{definition}

Both $\Mfld_G$ and $\D_G$ are symmetric monoidal under disjoint union.

\begin{definition}
Let $F:\D_G\to \C$ be a functor. 
\begin{enumerate}
\item We say that $F$ is \emph{locally constant} if $F$ sends isotopy equivalences to equivalences in $\C$.
\item We say that $F$ is \emph{reduced} if $F(\varnothing)\simeq0$. 
\end{enumerate}
\end{definition}

The definition of local constancy extends in the obvious way to functors $F:\Euc_G\to \C$. Note that a morphism in $\D_G$ is an isotopy equivalence if and only if it is a $\pi_0$-bijection; in particular, every morphism in $\Euc_G$ is an isotopy equivalence.

The following result asserts that we may work with the discrete category $\D_G$ without loss of information. This result is essentially \cite[5.4.5.9]{Lurie:HA}, and we shall employ the language of that reference for the duration of its proof.

\begin{proposition}\label{prop:discrete model}
Let $\C$ be a symmetric monoidal $\infty$-category. There is a commuting diagram \[\xymatrix{\Alg_{\mathbb{E}_G}^\mathrm{nu}(\C)\ar[d]\ar[r]&\Fun^\otimes(\D_G,\C)\ar[d]\\
\Mod_G(\C)\ar[r]&\Fun(\Euc_G,\C)
}\] in which both horizontal arrows are fully faithful with essential image the locally constant functors.
\end{proposition}
\begin{proof}
Letting $\EEuc_G\subseteq \MMfld_G$ denote the full subcategory spanned by the objects of $\Euc_G$, the functor $\Euc_G\to \EEuc_G$ is localization at the set of isotopy equivalences. Combining this fact with the equivalence $\EEuc_G\simeq BG$ of \cite[2.8]{AyalaFrancis:FHTM} supplies the bottom functor, its fully faithfulness, and the identification of the essential image.

We turn now to the top equivalence. Letting $\mathbb{E}^\delta_G$ denote the endomorphism operad of $\mathbb{R}^n$ in $\Mfld_G$, we note that $\D_G$ is equivalent to the symmetric monoidal envelope (see \cite[2.2.4]{Lurie:HA}) of the $\infty$-operad $(\mathbb{E}^{\delta}_G)_{\mathrm{nu}}$ controlling nonunital $\mathbb{E}_G^\delta$-algebras; thus, in light of the equivalences $B\Euc_G\simeq BG\simeq (\mathbb{E}_G)^\otimes_{\langle 1\rangle}$, it will suffice by \cite[2.3.3.23]{Lurie:HA} to verify that $\mathbb{E}^\delta_G\to \mathbb{E}_G$ is an approximation of $\infty$-operads in the sense of \cite[2.3.3]{Lurie:HA}, which follows by the argument of \cite[5.4.5.11]{Lurie:HA}.
\end{proof}

\subsection{Globalization} We turn now to the relation of these notions to the larger category $\Mfld_G$. We make the following obvious generalization:

\begin{definition}
We say that a functor $F:\Mfld_G\to \C$ is \emph{locally constant} if $F$ sends isotopy equivalences to equivalences in $\C$. 
\end{definition}

We write $\Fun^\loc(\Mfld_G,\C)\subseteq \Fun(\Mfld_G,\C)$ for the full subcategory spanned by the locally constant functors. Denoting the inclusion by $\jmath:\D_G\to \Mfld_G$, then it is obvious that $\jmath^\natural F$ is locally constant whenever $F$ is. Less obviously, we have the following result, whose proof is deferred momentarily for the sake of continuity.

\begin{lemma}\label{lem:pushforward local constancy}
If $F:\D_G\to \C$ is locally constant, then so is $\jmath_!F$. 
\end{lemma}

Thus, the $(\jmath_!,\jmath^\natural)$-adjunction descends to an adjunction between subcategories of locally constant functors. This adjunction also interacts well with the symmetric monoidal structures at hand; indeed, \cite[2.15]{AyalaFrancis:FHTM} implies that the $(\jmath_!,\jmath^\natural)$-adjunction lifts to an adjunction between $\infty$-categories of symmetric monoidal functors, i.e., that we have a diagram

\[\xymatrix{\Fun^\otimes(\D_G,\C)
\ar[d]\ar@/^1pc/[r]&\Fun^\otimes(\Mfld_G,\C)\ar[l]\ar[d]\\
\Fun(\D_G,\C)\ar@/^1pc/[r]^-{\jmath_!}&\Fun(\Mfld_G,\C)\ar[l]^-{\jmath^\natural}}\] in which both horizontal pairs are adjoint pairs and both square diagrams commute. 

\begin{definition}
Let $A$ be a nonunital $\mathbb{E}_G$-algebra in $\C$. \emph{Factorization homology with coefficients in $A$}, denoted $\int A$, is the value on $A$ of the composite \[\xymatrix{\Alg_{\mathbb{E}_G}^\mathrm{nu}(\C)\ar@{^{(}->}[r]^-{(\ref{prop:discrete model})}&\Fun^\otimes(\D_G,\C)\ar[r]^-{\jmath_!}& \Fun^\otimes(\Mfld_G,\C).}\]
\end{definition}

Note that $\int A$ is locally constant by Lemma \ref{lem:pushforward local constancy}. Moreover, the factorization homology retains all of the information of $A$; indeed, $\int_{\mathbb{R}^n}A\simeq A\oplus 1_\C$ as nonunital $\mathbb{E}_G$-algebras.

\begin{remark}
Factorization homology of the nonunital $\mathbb{E}_G$-algebra $A$ as defined here coincides with the factorization homology of the unital $\mathbb{E}_G$-algebra $A\oplus 1_\C$ as defined in \cite{AyalaFrancis:FHTM} (essentially, this identification follows from Proposition 2.19 of that work). The reader is cautioned not to conflate this with the version of factorization homology for nonunital disk algebras defined in \cite[5.5.4]{Lurie:HA}, denoted there by $\int^{\mathrm{nu}}A$. These two objects do not coincide.
\end{remark}

\begin{proof}[Proof of Lemma \ref{lem:pushforward local constancy}]
We have a factorization of $\jmath$ as the composition \[\D_G\xrightarrow{\jmath_1} \Disk_G\xrightarrow{\jmath_2} \Mfld_G,\] where $\Disk_G\subseteq \Mfld_G$ is the full subcategory spanned by the objects of $\D_G$; that is, morphisms in $\Disk_G$ need not surject on $\pi_0$. The lemma will follow upon verifying that $(\jmath_1)_!$ and $(\jmath_2)_!$ each preserve local constancy. 

For the first claim, we note that $\jmath_{1/U_I}$ receives a final functor from the discrete category of subsets of $I$, which sends $S\subseteq I$ to the object $(U_S, U_S\to U_I)$. Therefore, we have the natural equivalence \[(\jmath_1)_!F(U_I)\simeq \bigoplus_{S\subseteq I} F(U_S),\] which implies the local constancy of $(\jmath_1)_!F$.

For the second claim, let $F':\Disk_G\to \C$ be locally constant, let $\varphi:M\to N$ be an isotopy equivalence, and consider the commuting diagram \[\xymatrix{
\Disk_{G/M}\ar[d]_-{l_M}\ar[r]^\varphi&\Disk_{G/N}\ar[d]^-{l_N}\\
\DDisk_{G/M}\ar[r]^-{\tilde\varphi}&\DDisk_{G/N},
}\] where $\DDisk_G\subseteq \MMfld_G$ is the full subcategory spanned by the objects of $\Disk_G$. Denote by $F'_M$ the composite $\Disk_{G/M}\to \Disk_G\xrightarrow{F'} \C$, and note that $\varphi^\natural F'_N\simeq F'_M$. According to \cite[2.19]{AyalaFrancis:FHTM}, the functor $l_N$ is localization at the set of isotopy equivalences, so we have a factorization $F_N'\simeq l_N^\natural F_N''$ by local constancy. Thus,
\begin{align*}
(\jmath_2)_!F'(M)&\simeq \colim F_M'\\
&\simeq \colim\varphi^\natural F_N'\\
&\simeq \colim \varphi^\natural l_N^\natural F_N''\\
&\simeq \colim l_M^\natural \tilde \varphi^\natural F_N''\\
&\simeq \colim\tilde\varphi^\natural F_N''\\
&\simeq\colim F_N''\\
&\simeq \colim l_N^\natural F_N''\\
&\simeq \colim F_N'\\
&\simeq (\jmath_2)_!F'(N),
\end{align*} where the fifth equivalence follows from the fact that $l_M$ is a localization and hence final, the sixth from the fact that $\widetilde\varphi$ is an equivalence and hence final, and the seventh from the fact that $l_N$ is a localization and hence final.
\end{proof}

\subsection{The absolute Ran category}\label{sec:Ran section}

Let $M$ be an $n$-manifold. The \emph{Ran space} of $M$ is a topological space whose underlying set is the set \[\Ran(M)=\{S\subseteq M: 0<|S|<\infty\}\] of finite subsets of $M$. Following \cite[5.5.1]{Lurie:HA}, we topologize $\Ran(M)$ as follows. Let $\D(M)$ denote the partially ordered set of nonempty open subsets $U_I\subseteq M$ with $U_i\cong\mathbb{R}^n$ and inclusions that surject on $\pi_0$. We equip $\Ran(M)$ with the topology generated by $\{S\subseteq U_I: S\cap U_i\neq \varnothing,\, i\in I\}_{U_I\in \D(M)}$. Thus, in particular, we have a functor \[\D(M)\to \mathrm{Op}(\Ran(M)).\]

\begin{remark}
In \cite[5.5.2]{Lurie:HA}, the poset $\D(M)$ is denoted $\Disj(M)^{\mathrm{nu}}$.
\end{remark}

The reader interested in further information regarding the space $\Ran(M)$ may consult \cite{Tuffley:FSSS}, \cite[3.5.1]{BeilinsonDrinfeld:CA}, \cite[5.5]{Lurie:HA}, or \cite[3.7]{AyalaFrancisTanaka:LSSS}. Except in the proof of Lemma \ref{lem:matched extension} below, the main role in what follows of the Ran space as such will be as a guiding analogy in our exploration of the structural features of a certain category, which we now define.

\begin{definition}
The ($G$-framed) \emph{absolute Ran category} is the category $\Ran_G$ specified as follows:
\begin{enumerate}
\item an object of $\Ran_G$ is a $G$-framed manifold $M$ and a submanifold $U_I\subseteq M$ with $U_i\cong\mathbb{R}^n$;
\item a morphism from $U_I\subseteq M$ to $V_J\subseteq N$ is a $G$-framed embedding $\varphi:M\to N$ such that 
\begin{enumerate}
\item $\varphi(U_I)\subseteq V_J$ and 
\item $\pi_0(\varphi|_{U_I}):I\to J$ is surjective.
\end{enumerate}
\end{enumerate}
\end{definition}

Forgetting the submanifold defines a functor $\pi:\Ran_G\to \Mfld_G$, while equipping $U_I\subseteq M$ with the $G$-framing induced by $M$ defines a functor $\epsilon:\Ran_G\to \D_G$.

\begin{remark}\label{rmk:Ran fiber}
Our choice of terminology is motivated by the observation that there is an isomorphism $\pi^{-1}(M)\cong \D(M)\amalg\{\varnothing\}$. Since $\D(M)$ is a basis for the topology of $\Ran(M)$, we think of $\Ran_G$ as something like a bundle over the moduli (category) of $G$-framed manifolds with fiber over $M$ the Ran space of $M$ (with the addition of a disjoint basepoint $\{\varnothing\}$).
\end{remark}

We shall say that a morphism $\varphi\in \Ran_G$ is an isotopy equivalence if $\epsilon(\varphi)\in\D_G$ is an isotopy equivalence. The reader is cautioned that, in this case, $\pi(\varphi)$ is typically \emph{not} an isotopy equivalence.

\begin{definition}\label{def:constructible cosheaf}
Let $F:\Ran_G\to \C$ be a functor. 
\begin{enumerate}
\item We say that $F$ is a \emph{constructible cosheaf} if $F$ sends isotopy equivalences to equivalences in $\C$.
\item We say that $F$ is \emph{reduced} if $F(\varnothing\subseteq \varnothing)\simeq0$. 
\end{enumerate}
\end{definition}

We write $\cShv^\cbl(\Ran_G,\C)\subseteq\Fun(\Ran_G,\C)$ for the full subcategory spanned by the constructible cosheaves (resp. $\cShv_0^\cbl(\Ran_G,\C)$, reduced constructible cosheaves).

\begin{remark}
If $F:\Ran_G\to \C$ is a constructible cosheaf, then, according to \cite[5.5.1.14,\,5.5.4.13]{Lurie:HA} and Remark \ref{rmk:Ran fiber}, the restriction of $F$ to the fiber $\pi^{-1}(M)\cong\D(M)\amalg\{\varnothing\}$ determines a $\C$-valued cosheaf on the topological space $\Ran(M)_+$ that is constructible with respect to the natural stratification by cardinality. Thus, the data of a reduced constructible cosheaf on $\Ran_G$ amounts to a functorial choice of constructible cosheaf on the Ran space of each $G$-manifold separately. 
\end{remark}

\begin{lemma}\label{lem:cosheaf identification}
Restriction along $\epsilon:\Ran_G\to \D_G$ induces the horizontal equivalences in the commuting diagram \[\xymatrix{
\Fun^\loc(\D_G, \C)\ar[r]^-\sim&\cShv^\cbl(\Ran_G,\C)\\
\Fun_0^\loc(\D_G, \C)\ar[u]\ar[r]^-\sim&\cShv_0^\cbl(\Ran_G,\C)\ar[u]
}\]
\end{lemma}
\begin{proof}
Fix $F_1\in \Fun^\loc(\D_G,\C)$ and $F_2\in \cShv^\cbl(\Ran_G,\C)$. The lemma follows from the following five claims.
\begin{enumerate}
\item \emph{The restriction $\epsilon^\natural F_1$ is a constructible cosheaf.} The claim is immediate from the definitions.
\item \emph{The right Kan extension $\epsilon_*F_2$ is locally constant.} By Lemma \ref{lem:initial object} below, we have the natural equivalence $\epsilon_*F_2(U_I)\simeq F_2(U_I\subseteq U_I)$, and the claim follows by constructibility. 
\item \emph{The unit $F_1\to \epsilon_*\epsilon^\natural F_1$ is an equivalence}. The claim is immediate from Lemma \ref{lem:initial object}.
\item \emph{The counit $\epsilon^\natural\epsilon_*F_2\to F_2$ is an equivalence}. By Lemma \ref{lem:initial object}, the value of the counit at $U_I\subseteq M$ is the value of $F_2$ on the morphism $U_I\subseteq U_I\to U_I\subseteq M$, which is an equivalence by constructibility. 
\item \emph{$F_1$ is reduced if and only if $\epsilon^\natural F_1$ is reduced}. The claim is immediate from the definitions.
\end{enumerate}
\end{proof}

\begin{lemma}\label{lem:initial object}
The object $(U_I\subseteq U_I, \id_{U_I})$ is initial in $U_I\downarrow \epsilon$
\end{lemma}
\begin{proof}
Given an object $V_J\subseteq M$ of $\Ran_G$ and a morphism $\varphi:U_I\to V_J$ in $\D_G$, the composite $U_I\xrightarrow{\varphi} V_J\subseteq M$ defines a morphism from $U_I\subseteq U_I$ to $V_J\subseteq M$ in $\Ran_G$ lifting $\varphi$. Since a $G$-framed embedding is determined by the restriction of its codomain to its image, this lift is unique.
\end{proof}

We obtain a functor \[\pi^!:=\epsilon^\natural\jmath^\natural:\Fun^\loc(\Mfld_G,\C)\to \cShv^\cbl(\Ran_G,\C).\] Note that $\pi^!$ is \emph{not} the restriction of $\pi^\natural$ to the subcategory of locally constant functors; indeed, $\pi^\natural F$ is almost never a constructible cosheaf. Our choice of notation is justified by the following lemma, which guarantees that the suggested adjunction \[\pi_!:\cShv^\cbl(\Ran_G,\C)\rightleftarrows \Fun^\loc(\Mfld_G,\C):\pi^!\] does in fact exist. 

\begin{lemma}
If $F:\Ran_G\to \C$ is a constructible cosheaf, then $\pi_!F$ is locally constant.
\end{lemma}
\begin{proof}
The functor $\epsilon$ admits a section $\sigma:\D_G\to \Ran_G$, given by $\sigma(U_I)=U_I\subseteq U_I$, and clearly $\pi\sigma=\jmath$. The equivalence $\epsilon_!\sigma_!\simeq 1$ yields a natural transformation $\sigma_!\to \epsilon^\natural$ by adjunction, whence a natural transformation \[\jmath_!\simeq \pi_!\sigma_!\to \pi_!\epsilon^\natural.\] Unwinding the definitions, the component of this map on the object $M$ is the map induced on colimits by the functor $\jmath_{/M}\to \pi_{/M}$ induced by $\sigma$. To show this map is an equivalence, we note that it factors as a composite of functors \[\jmath_{/M}\to \pi^{-1}(M)\to \pi_{/M},\] both of which are final. But now the proof is complete, for \[\pi_!F\overset{(\ref{lem:cosheaf identification})}{\simeq} \pi_!\epsilon^\natural \epsilon_!F\simeq \jmath_!\epsilon_!F\] is locally constant by Lemma \ref{lem:pushforward local constancy}, since $\epsilon_!F$ is locally constant by Lemma \ref{lem:cosheaf identification}.
\end{proof}

Within $\Ran(M)$ lies a closed subspace naturally identified with the manifold $M$, namely the subspace of singletons. The corresponding object in our context is the following:

\begin{definition}
The ($G$-framed) \emph{absolute diagonal} is the full subcategory $\Diag_G\subseteq \Ran_G$ spanned by the objects $U_I\subseteq M$ with $|I|=1$. 
\end{definition}

\begin{definition}
A \emph{locally constant cosheaf} (resp. sheaf) on the absolute diagonal is a functor $F:\Diag_G\to \C$ (resp. $\Diag_G^{op}$) sending every arrow to an equivalence in $\C$.
\end{definition}

We write $\cShv^\loc(\Diag_G,\C)\subseteq \Fun(\Diag_G,\C)$ and $\Shv^\loc(\Diag_G,\C)\subseteq \Fun(\Diag_G^{op},\C)$ for the full subcategories spanned by the locally constant cosheaves and the locally constant sheaves, respectively. 

\begin{remark}
The fiber over $M$ of the projection $\Diag_G\to \Mfld_G$ is the poset $\Euc(M)$ of Euclidean neighborhoods in $M$. Since $B\Euc(M)\simeq M$, we see that the data of a locally constant (co)sheaf on the absolute diagonal amounts to a functorial choice of locally constant (co)sheaf on each $G$-manifold separately. 
\end{remark}

Classically, one knows that the compactly supported sections of a sheaf $F$ assemble into a cosheaf $F_c$, and, in good situations, this correspondence between sheaves and cosheaves is an equivalence, which goes by the name of \emph{Verdier duality}. We have the following analogue in the absolute context:

\begin{lemma}\label{lem:verdier}
Restriction along $\epsilon:\Diag_G\to \Euc_G$ induces a commuting diagram of equivalences \[\xymatrix{
\Mod_G(\C)\simeq \Fun^\loc(\Euc^{op}_G,\C)\ar[d]^-\wr_-{\Omega^n}\ar[r]^-\sim& \Shv^\loc(\Diag_G,\C)\ar[d]_-\wr^-{(-)_c}\\
\Mod_G(\C)\simeq \Fun^\loc(\Euc_G,\C)\ar[r]^-\sim& \cShv^\loc(\Diag_G,\C).}\]
\end{lemma}

\subsection{Push and pull}\label{sec:push pull section}

In this section, we study various functorialities for constructible cosheaves, guided by the analogy between $\Ran_G$ and $\Ran(M)$. The pushforward and pullback functors introduced here will form the basis for our later arguments; for the sake of continuity, we defer their construction to \S\ref{sec:functoriality section} below. 

A map of fundamental importance to the Ran space of a manifold $M$ is the so-called \emph{main diagonal} \begin{align*}
M&\longrightarrow \Ran(M)\\
x&\longmapsto \{x\}.
\end{align*} The corresponding functor in the absolute context is the inclusion $\delta:\Diag_G\to \Mfld_G$. The restriction and left Kan extension adjunction corresponding to this functor induces an adjunction \[\delta_!:\cShv^\loc(\Diag_G,\C)\rightleftarrows \cShv^\cbl(\Ran_G,\C):\delta^!.\] There is also an exceptional pullback $\delta^*$ and an adjunction \[\delta^*:\cShv^\cbl(\Ran_G,\C)\rightleftarrows \cShv^\loc(\Diag_G,\C):\delta_!.\] For more on these functors, see \S\ref{sec:main diagonal}.

\begin{definition}
We say that a constructible cosheaf $F:\Ran_G\to \C$ is \emph{supported on the diagonal} if $F(U_I\subseteq M)\simeq0$ whenever $|I|\neq 1$. 
\end{definition}

According to Corollary \ref{cor:diagonal support}, the pushforward $\delta_!$ is fully faithful with essential image the constructible cosheaves supported on the diagonal.

The Ran space of a manifold $M$ is naturally a (nonunital) topological monoid under the multiplication \begin{align*}
\Ran(M)^2&\longrightarrow \Ran(M)\\
(S,T)&\longmapsto S\cup T.
\end{align*} In order to formulate the corresponding push and pull functors, we make the following obvious definition:

\begin{definition}
Let $F:\Ran_G\times_{\Mfld_G}\Ran_G\to \C$ be a functor. 
\begin{enumerate}
\item We say that $F$ is a \emph{constructible cosheaf} if $F$ sends pairs of isotopy equivalences to equivalences in $\C$.
\item We say that $F$ is \emph{reduced} if $F(\varnothing\subseteq \varnothing, -)\simeq F(-,\varnothing\subseteq \varnothing)\simeq 0$. 
\end{enumerate}
\end{definition}

Although the multiplication introduced above has no obvious incarnation as a functor $\Ran_G\times_{\Mfld_G}\Ran_G\to \Ran_G$, we nevertheless have an adjunction \[\mu_!:\cShv^\cbl(\Ran_G\times_{\Mfld_G}\Ran_G,\C)\rightleftarrows \cShv^\cbl(\Ran_G,\C):\mu^!.\] It turns out that $\mu_!$ is also the \emph{right} adjoint to $\mu^!$, a fact corresponding to the properness of the multiplication map $\Ran(M)^2\to \Ran(M)$, and we shall at times use the alternate notation $\mu_*$ when we wish to emphasize the role of this pushforward as a right adjoint. The functor $\mu_!$ obeys the formula \[\mu_!F(U_I\subseteq M)\simeq \bigoplus_{I_1\cup I_2=I}F(U_{I_1}\subseteq M, U_{I_2}\subseteq M).\] In particular, if $F$ is reduced, then so is $\mu_!F$. For more on these functors, see \S\ref{sec:matched disks}.

As for the absolute diagonal, the main piece of functoriality relevant for our purposes concerns the map $\Delta:\Diag_G\to \Diag_G\times_{\Mfld_G}\Diag_G$. With the evident extension of the notion of locally constant cosheaf to $\Diag_G\times_{\Mfld_G}\Diag_G$, we have as before that the $(\Delta_!,\Delta^\natural)$-adjunction descends to an adjunction \[\Delta_!:\cShv^\loc(\Diag_G,\C)\rightleftarrows \cShv^\loc(\Diag_G\times_{\Mfld_G}\Diag_G,\C):\Delta^!,\] and that $\Delta_!$ admits a right adjoint $\Delta^*$. For more on these functors, see \S\ref{sec:base change}.

These two pullback functors induce two tensor products on locally constant cosheaves \[F\otimes^! G=\Delta^!(F\boxtimes G)\qquad\qquad\text{and}\qquad\qquad F\otimes^* G=\Delta^*(F\boxtimes G),\] where $\boxtimes$ denotes the external tensor product. Through the equivalence of Lemma \ref{lem:verdier}, the former, pointwise tensor product corresponds to the ordinary tensor product of $G$-modules, while the latter corresponds to the tensor product \[M\otimes^*N\simeq \Omega^n(\Sigma^nM\otimes\Sigma^nN)\] for $G$-modules $M$ and $N$.

\begin{remark}
The reader is cautioned not to conflate the tensor product $\otimes^*$ with the star tensor product $\otimes^\star$ of \cite{BeilinsonDrinfeld:CA} and \cite{FrancisGaitsgory:CKD}. The analogue of the latter in our context is the \emph{overlapping tensor product} $\otimes^\cup$ (see \S\ref{sec:monoidal section}). 
\end{remark}

The functors introduced in this section interact as predicted by the topological analogy. For example, since the square \[\xymatrix{
M\ar@{=}[d]\ar[r]&M\times M\ar[r]&\Ran(M)\times\Ran(M)\ar[d]\\
M\ar[rr]&&\Ran(M)
}\] is a pullback and the maps involved are proper, we should expect a base change result to hold. A precise statement of this result, whose proof may be found in \S\ref{sec:base change}, is the following:

\begin{lemma}\label{lem:base change}
The canonical map $\delta^*\mu_!\to \Delta^*(\delta\times\delta)^*$ is an equivalence.
\end{lemma}

This observation will be a key ingredient in Proposition \ref{prop:overlapping lie} and thereby in the main theorems.

\subsection{Factorizable coalgebras}\label{sec:factorizable coalgebras} In this section, we arrive at the promised cocommutative coalgebra model for nonunital $\mathbb{E}_G$-algebras; see Proposition \ref{prop:factorizable 2} below. Prerequisitely, we introduce the following symmetric monoidal structure:

\begin{proposition}\label{prop:disjoint day}
Let $\C$ be a stable, presentably symmetric monoidal $\infty$-category. There is a symmetric monoidal structure on $\Fun(\D_G,\C)$, called \emph{right Day convolution}, with tensor product given by the formula \[\bigg(\bigotimes^\mathrm{RD}_JF_j\bigg)(U_I)\simeq \bigoplus_{\Fin(I,J)}\bigotimes_J F_j(U_{f^{-1}(j)})\] and an equivalence of $\infty$-categories \[\Fun^\mathrm{oplax}(\D_G,\C)\simeq \Coalg_\Com(\Fun(\D_G,\C))\] covering the identity on $\Fun(\D_G,\C)$.
\end{proposition}

The proof of this proposition will involve a detour through the formalism of Day convolution and is deferred to \S\ref{sec:day convolution}; however, it is worth remarking that, under the asserted equivalence, the components of the oplax structure on a functor $F$ are identified with the components of the comultiplication on the corresponding coalgebra. Since our interest, through the equivalence of Proposition \ref{prop:discrete model}, lies in the oplax functors that happen to be symmetric monoidal, it will be useful to give a name to the criterion guaranteeing that the a cocommutative coalgebra corresponds to such a functor.

\begin{definition}
We say that $A\in\Coalg_\Com(\Fun(\D_G,\C))$ is \emph{factorizable} if the composite \[A(U_I)\longrightarrow \bigotimes^\mathrm{RD}_IA(U_I)\simeq \bigoplus_{\Fin(I,I)}\bigotimes_IA(U_{f^{-1}(i)})\xrightarrow{\pi_{\id_I}} \bigotimes_I A(U_i)\] is an equivalence for every $U_I\in \D_G$. 
\end{definition}

\begin{remark}
This condition is equivalent to the apparently stronger requirement that the composite \[A(U_I)\to \bigotimes^\mathrm{RD}_JA(U_I)\simeq\bigoplus_{\Fin(I,J)}\bigotimes_JA(U_{f^{-1}(j)})\xrightarrow{\pi_f} \bigotimes_JA(U_{f^{-1}(j)})\] be an equivalence for every $U_I\in\D_G$, every $J$, and every $f:I\to J$.
\end{remark}

\begin{remark}\label{rem:factorizable unit}
If $A$ is factorizable, then $A(\varnothing)\simeq 1_\C$.
\end{remark}

Denoting by $\Fact(\Fun(\D_G,\C))\subseteq\Coalg_\Com(\Fun(\D_G,\C)_\mathrm{RD})$ the full subcategory spanned by the factorizable coalgebras, we have the following:

\begin{corollary}\label{cor:factorizable 1} The dashed equivalence exists in the commuting diagram
\[\xymatrix{\Fun^\otimes(\D_G,\C)\ar[d]\ar@{-->}[r]^-\sim&\Fact(\Fun(\D_G,\C))\ar[d]\\
\Fun^{\oplax}(\D_G,\C)\ar[r]^-\sim&\Coalg_\Com(\Fun(\D_G,\C)_\mathrm{RD}).
}\]
\end{corollary}

From the formula of Proposition \ref{prop:disjoint day}, we see that the tensor product of locally constant functors is again locally constant; therefore, through the equivalence of Lemma \ref{lem:cosheaf identification}, we obtain a tensor product on constructible cosheaves. We refer to this monoidal structure as the \emph{disjoint monoidal structure}. Where necessary, the symbol $\amalg$ will be used to disambiguate the disjoint tensor product from other tensor products (e.g. $\otimes^\amalg$).

The formula also makes it clear that the tensor product preserves the property of being reduced, so that $\Fun^\loc_0(\D_G,\C)$ and $\cShv_0^\cbl(\Ran_G,\C)$ inherit nonunital symmetric monoidal structures. We note that, for reduced functors $F_J$, we have the formula \[\bigg(\bigotimes_J^\amalg F_j\bigg)(U_I)\simeq \bigoplus_{\Surj(I,J)} \bigotimes_J F_j(U_{\pi^{-1}(j)}),\] which should be compared to the formula for the value on stalks of the chiral tensor product of \cite[\text{p. 4}]{FrancisGaitsgory:CKD}.

We shall say that a cocommutative coalgebra in $\cShv^\cbl(\Ran_G,\C)$ is factorizable if the corresponding coalgebra in $\Fun(\D_G,\C)$ is, and we shall say that a noncounital cocommutative coalgebra in $\cShv^\cbl_0(\Ran_G,\C)$ is factorizable if its counitalization is. We extend the notation $\Fact$ to these cases in the obvious way.

The main result of this section is the following:

\begin{proposition}\label{prop:factorizable 2}
There is an equivalence \[\Alg^\mathrm{nu}_{\mathbb{E}_G}(\C)\simeq \Fact(\cShv_0^\cbl(\Ran_G,\C)).\]
\end{proposition}
\begin{proof}
From the definitions, we have the dashed factorization in the diagram
\[\xymatrix{\Coalg_\Com^\mathrm{nu}(\Fun^\loc_0(\D_G,\C)_\mathrm{RD})\ar@{^{(}->}[r]&\Coalg_\Com^\mathrm{nu}(\Fun^\loc(\D_G,\C)_\mathrm{RD})\ar[r]^-{(-)\oplus 1_{\mathrm{RD}}}&\Coalg_\Com(\Fun^\loc(\D_G,\C)_\mathrm{RD})\\
\Coalg^\mathrm{nu}_\Com(\cShv_0^\cbl(\Ran_G,\C)_\amalg)\ar[u]^-\wr&\Fact(\cShv_0^\cbl(\Ran_G,\C))\ar@{_{(}->}[l]\ar@{-->}[r]&\Fact(\Fun^\loc(\D_G,\C)),\ar[u]
}\] where \[1_{\mathrm{RD}}(U_I)\simeq\begin{cases}
1_\C&\quad I=\varnothing\\
0&\quad\text{else}
\end{cases}\] is the unit for $\otimes^{\mathrm{RD}}$. It will suffice to show that this functor is an equivalence, since we have the equivalence $\Alg_{\mathbb{E}_G}(\C)\simeq \Fact(\Fun^\loc(\D_G,\C))$ by imposing local constancy on Corollary \ref{cor:factorizable 1} and invoking Proposition \ref{prop:discrete model}.

Now, $1_{\mathrm{RD}}$ is factorizable and, in fact, initial in $\Fact(\Fun^\loc(\D_G,\C))$; indeed, it is the factorizable coalgebra corresponding to the trivial algebra $0\in\Alg_{\mathbb{E}_G}^\mathrm{nu}(\C)$, which is initial. Thus, we obtain a fully faithful factorization of the inclusion as indicated in the following diagram: \[\xymatrix{
\Fact(\Fun^\loc(\D_G,\C))\ar@{^{(}-->}[d]\ar@{^{(}->}[r]&\Coalg_\Com(\Fun^\loc(\D_G,\C))\\
\Coalg^\mathrm{coaug}_\Com(\Fun^\loc(\D_G,\C)).\ar[ur]
}\] Since adjoining the counit induces an equivalence between noncounital and coaugmented coalgebras, this implies that the dashed functor is fully faithful. To see that it is also essentially surjective, it suffices to note that the coaugmentation coideal of $A$ is reduced and factorizable by Remark \ref{rem:factorizable unit}.
\end{proof}

Now, the argument given below in \S\ref{sec:day convolution} in the case of $\D_G$ shows that right Day convolution exists on $\Fun(\Mfld_G,\C)$ and obeys the same formula. Because the functor $\jmath$ is symmetric monoidal, $\jmath^\natural$ obtains an oplax structure for right Day convolution, and, from the explicit formula \begin{align*}\jmath^\natural(F\otimes^{\mathrm{RD}}G)(U_I)&=(F\otimes^{\mathrm{RD}}G)(\jmath(U_I))\\&\simeq \bigoplus_{I_1\sqcup I_2=I, I_j\neq\varnothing}F(\jmath(U_{I_1}))\otimes G(\jmath(U_{I_2}))\\&=\bigoplus_{I_1\sqcup I_2=I, I_j\neq\varnothing}\jmath^\natural F(U_{I_1})\otimes \jmath^\natural G(U_{I_2})\\&\simeq (\jmath^\natural F\otimes^\mathrm{RD}\jmath^\natural G)(U_I), \end{align*} we see that this oplax monoidal structure is in fact strong monoidal. Thus, as the left adjoint of a strong monoidal functor, $\jmath_!$ inherits an oplax structure, and the $(\jmath_!,\jmath^\natural)$-adjunction lifts to an adjunction between the respective $\infty$-categories of cocommutative coalgebras, which we abusively indicate by the same symbols.

\begin{lemma}\label{lem:preserves factorizability}
If $A\in \Coalg_\Com(\Fun(\D_G,\C))$ is factorizable, then so is $\jmath_!A$.
\end{lemma}
\begin{proof}
For $U_I\in \Mfld_G$, the value of the left Kan extension $\jmath_!A$ on $U_I$ is computed as the colimit over the overcategory $\jmath_{/U_I}$ \cite[4.3.2]{Lurie:HTT}. By inspection, sending an object $V_J\to U_I$ with underlying map of finite sets $f:J\to I$ to the tuple $(V_{f^{-1}(i)}\to U_i)_{i\in I}$ determines an equivalence of categories $\jmath_{/U_I}\simeq \prod_{I}\jmath_{/U_i}$. Using this observation, we compute that \begin{align*}\jmath_!A(U_I)&\simeq \colim_{\jmath_{/U_I}}A\\&\simeq \colim_{\prod_I\jmath_{/U_i}}A\\
&\simeq \colim_{\prod_I\jmath_{/U_i}}\bigotimes_I A\\
&\simeq \bigotimes_I\colim_{\jmath_{/U_i}}A\\&\simeq \bigotimes_I \jmath_!A(U_i),\end{align*} where for the third equivalence we have used our assumption that $A$ is factorizable to obtain the equivalences $A(\amalg_I V_{f^{-1}(i)})\simeq \bigotimes_IA(V_{f^{-1}(i)})$ naturally in $V$ and $f$, and for the fourth equivalence we have used that $\otimes$ distributes over colimits, since $\C$ is presentably symmetric monoidal.
\end{proof}

\subsection{Nilpotence}\label{sec:nilpotence} In this section, following \cite[5.1.2]{FrancisGaitsgory:CKD}, we establish an excellent property of the disjoint monoidal structure, its pro-nilpotence.

\begin{definition}[\cite{FrancisGaitsgory:CKD}]\label{nilpotence definition}
Let $\V$ be a nonunital presentably symmetric monoidal stable $\infty$-category. We say that $\V$ is \emph{pro-nilpotent} if it can be exhibited as a limit $$\xymatrix{\V\ar[r]^-\sim& \displaystyle\lim_{\mathbb{N}^\op}\V_i}$$ of nonunital presentably symmetric monoidal stable $\infty$-categories such that \begin{enumerate}
\item $\V_0=\pt$,
\item for every $i\geq j$, the transition functor $f_{i,j}:\V_i\to \V_j$ commutes with limits and colimits, and
\item for every $i$, the restriction of the tensor product to $\ker(f_{i,i-1})\otimes \V_i$ is null.
\end{enumerate}
We say that $\V$ is \emph{nilpotent of order $r$} if $f_{i,j}$ is an equivalence for $i,j\geq r$.
\end{definition}

\begin{remark}
Using (3) inductively with (1) as a base case, it is easy to see that $r+1$-fold tensor products vanish in $\V$ when $\V$ is nilpotent of order $r$.
\end{remark}

To establish this property in the case at hand, we will exploit the filtration of $\Ran_G$ by cardinality.

\begin{definition}
The $k$-\emph{truncated absolute Ran category} is the full subcategory $\Ran_{G, k}\subseteq \Ran_G$ spanned by the objects $U_I\subseteq M$ with $|I|\leq k$.
\end{definition}

We extend Definition \ref{def:constructible cosheaf} in the obvious way to obtain a notion of (reduced) constructible cosheaf on $\Ran_{G,k}$. 

\begin{lemma}
Let $f:F_1\to F_2$ be a morphism and $F$ an object in $\cShv^\cbl(\Ran_G,\C)$. If $f|_{\Ran_{G,k}}$ is an equivalence, then so is $(f\otimes^\amalg \id_{F})|_{\Ran_{G, k}}$.
\end{lemma}
\begin{proof}
Given an object $U_I\subseteq M$ with $|I|\leq k$, we have the commuting diagram $$\xymatrix{
(F_1\otimes^\amalg F)(U_I\subseteq M)\ar[d]_-{(f\otimes^\amalg F)(U_I\subseteq M)}\ar[r]^-\sim&\displaystyle\bigoplus_{I_1\amalg I_2=I}F_1(U_{I_1}\subseteq M)\otimes F(U_{I_2}\subseteq M)\ar[d]^-{\bigoplus f(U_{I_1}\subseteq M)\otimes \id_{F(U_{I_2}\subseteq M)}}\\
(F_2\otimes^\amalg G)(U_I\subseteq M)\ar[r]^-\sim&\displaystyle\bigoplus_{I_1\amalg I_2=I}F_2(U_{I_1}\subseteq M)\otimes G(U_{I_2}\subseteq M).
}$$ For each $I_1\amalg I_2=I$, we have $|I_1| \leq |I|\leq k$, so $f(U_{I_1}\subseteq M)$ is an equivalence by assumption. Hence the righthand vertical arrow is an equivalence, so the lefthand vertical must also be an equivalence by two-out-of-three.
\end{proof}

From \cite[2.2.1.9]{Lurie:HA}, we conclude that $\cShv^\cbl(\Ran_{G,k},\C)$ inherits a symmetric monoidal structure, and $\cShv^\cbl_0(\Ran_{G,k},\C)$ a nonunital symmetric monoidal structure, rendering the restriction from $\Ran_G$ symmetric monoidal. These successive localizations assemble to form an inverse system 
\[\xymatrix{
\cShv_0^\cbl(\Ran_G,\C)\ar[r]&\displaystyle\bigg(\cdots\ar[r]&\cShv_0^\cbl(\Ran_{k,G},\C)\ar[r]&\cShv_0^\cbl(\Ran_{G,k-1},\C)\ar[r]&\cdots\displaystyle\bigg)
}\] of nonunital symmetric monoidal $\infty$-categories. Since equivalences of nonunital symmetric monoidal $\infty$-categories are detected at the level of underlying $\infty$-categories, and since the natural map \[\colim_\mathbb{N}\Ran_{G,k}\to \Ran_G\] is an equivalence, we conclude that the induced map \[\xymatrix{\cShv_0^\cbl(\Ran_G,\C)\ar[r]^-\sim&\displaystyle\lim_{\mathbb{N}^\op}\cShv_0^\cbl(\Ran_{G,k},\C)}\] is an equivalence of nonunital symmetric monoidal $\infty$-categories.

\begin{proposition}\label{prop:chiral is nilpotent}
The disjoint monoidal structure on $\cShv_0^\cbl(\Ran_G,\C)$ is pro-nilpotent.
\end{proposition}
\begin{proof} Having exhibited $\cShv_0^\cbl(\Ran_G,\C)$ as a sequential limit, there are three points to verify. 
\begin{enumerate}
\item Every object of $\Ran_{G,0}$ is of the form $\varnothing\subseteq M$, so $\cShv_0^\cbl(\Ran_{G,0},\C)=\{0\}$ by constructibility and reduction.
\item The transition functor $\cShv_0^\cbl(\Ran_{G,i},\C)\to \cShv_0^\cbl(\Ran_{G,j},\C)$ is restriction along the inclusion $\Ran_{G,j}\to \Ran_{G,i}$, which is both a right and a left adjoint and hence commutes with limits and colimits.
\item Let $F_1$ and $F_2$ be objects of $\cShv_0^\cbl(\Ran_{G,k},\C)$ such that the restriction of $F_1$ to $\Ran_{G,k-1}$ is trivial, and let $U_I\subseteq M$ be such that $|I|\leq k$. We have the equivalence \[\xymatrix{(F_1\otimes^\amalg F_2)(U_I\subseteq M)\ar[r]^-\sim&\displaystyle\bigoplus_{I_1\amalg I_2=I,\,I_j\neq \varnothing}F_1(U_{I_1}\subseteq M)\otimes F_2(U_{I_2}\subseteq M).}\] Since $I_2\neq \varnothing$, we have $|I_1|<|I|\leq k$, so $F_1(U_{I_1}\subseteq M)\simeq0$ by our assumption on $F_1$. Hence each term of the above sum vanishes, as desired.
\end{enumerate}
\end{proof}

\section{Lie models} 

\subsection{Lie algebras}\label{sec:Lie section} It has long been known that the structure of a Lie algebra is controlled by an operad in graded abelian groups, the \emph{Lie operad}, whose $\Sigma_n$-module of arity $n$ operations is the homology of a certain partially ordered set of partitions (see \cite{Fresse:KDOHPP}). Recently, the concept of a Lie algebra has been lifted to the world of stable homotopy. These \emph{spectral Lie algebras} are algebras over the operad introduced in \cite{Ching:BCTOGDI}, denoted here by the letter $\Lie$, whose components are the Spanier-Whitehead duals of these same partition posets, and which, upon passing to $\mathbb{Z}$-modules, recovers the familiar Lie operad. At the time of writing, spectral Lie algebras are the subject of intense investigation; see \cite{AntolinCamarena:MTHFSLA}, \cite{Brantner:LTTSLA}, and \cite{Kjaer:OPHFAOSLO}, for example. 

Our passage from the cocommutative model of Proposition \ref{prop:factorizable 2} to the world of Lie algebras will rely on a fundamental relationship enjoyed by these structures. Since its discovery and spectacular exploitation by Quillen in the seminal paper \cite{Quillen:RHT}, this relationship of \emph{Koszul duality} has been studied intensively in increasingly general contexts---see \cite{GinzburgKapranov:KDO}, \cite{Hinich:DGCFS}, \cite{Ching:BCTOGDI}, \cite{FrancisGaitsgory:CKD}, and \cite{Francis:TCHCER}, for example. From a modern viewpoint, it would seem that Lie algebras should be thought of as being defined by this duality. 

The use of Koszul duality for spectral operads in a higher categorical context is well-established in the literature, but its foundations remain folklore at the time of writing, and it is beyond the scope of this paper to alter this state of affairs. We now summarize the precise version of the theory that we shall employ, which is gathered from \cite[3,\,4]{FrancisGaitsgory:CKD}.

Let $\V$ be a stable nonunital presentably symmetric monoidal $\infty$-category. The main player is the functor \[\overline C^\Lie:\Alg_\Lie(\V)\longrightarrow \Coalg^\mathrm{nu}_{\Com}(\V)\] of (reduced) \emph{Lie chains}. This functor has the following features:

\begin{enumerate}
\item there is a functorial filtration $\overline C^\Lie(L)\simeq \colim_\mathbb{N}\overline C^\Lie(L)_{\leq k}$ of cocommutative coalgebras with associated graded coalgebra $$\gr\,\overline C^\Lie(L):=\bigoplus_{k\geq1}\mathrm{cofib}\big(\overline C^\Lie(L)_{\leq k}\to \overline C^\Lie(L)_{\leq k+1}\big)\simeq \Sym(L[1]);$$
\item the composite of $\overline C^\Lie$ with the forgetful functor to $\V$ is naturally equivalent, after a shift of degree $-1$, to the monadic bar construction $B(\id, \Lie,-)$ against the free Lie algebra functor.
\end{enumerate}

When $\V$ is a unital symmetric monoidal $\infty$-category viewed as nonunital, we write \[C^\Lie:=\overline C^\Lie\oplus 1_\V\] for the corresponding coaugmented cocommutative coalgebra.

\begin{remark}
In the setting of chain complexes over a field of characteristic zero, $C^\Lie$ is modeled by the classical Chevalley-Eilenberg complex (see \cite{Fresse:KDOHPP} for a comparison).
\end{remark}

We have the following key result concerning Lie chains.

\begin{theorem}[Francis-Gaitsgory]\label{thm:duality}
Let $\V$ be a nonunital presentably symmetric monoidal stable $\infty$-category with the following properties: \begin{enumerate}
\item $\V$ is pro-nilpotent;
\item the norm map $(v^{\otimes k})_{\Sigma_k}\to (v^{\otimes k})^{\Sigma_k}$ is an equivalence for very $v\in\V$ and $k\in\mathbb{N}$.
\end{enumerate}
Then $\overline C^\Lie$ is an equivalence.
\end{theorem}
\begin{proof}
Following \cite[3.3]{FrancisGaitsgory:CKD}, the functor $\overline C^\Lie$ (denoted there by $\mathrm{Bar}_{\mathrm{Lie}}^{\mathrm{enh}}$) factors through the inclusion of the $\infty$-category of non-counital cocommutative coalgebras that are both conilpotent and equipped with a ``codivided copower'' structure, which we denote for the duration of this argument by $\D$. This inclusion is represented on cofree objects by the composite \[\xymatrix{\displaystyle\bigoplus_{k\geq1} (v^{\otimes k})_{\Sigma_k}\ar[r]&\displaystyle\prod_{k\geq1} (v^{\otimes k})_{\Sigma_k}\ar[r]& \displaystyle\prod_{k\geq1} (v^{\otimes k})^{\Sigma_k}}\] of the canonical map from coproduct to product with the product of the respective norm maps. According to \cite[4.1.2]{FrancisGaitsgory:CKD}, the assumption of pro-nilpotence guarantees that $\overline C^\Lie$ induces an equivalence $\Alg_\Lie(\V)\simeq \D$, so the claim will be proven upon verifying that $D\simeq \Coalg^\mathrm{nu}_\Com(C)$, for which it suffices to check that each of the arrows depicted above is an equivalence. The first equivalence follows from pro-nilpotence as in \cite[4.2.1]{FrancisGaitsgory:CKD}, while the second follows from our assumption (2).
\end{proof}

\begin{remark}
The observant reader will notice that the statement of \cite[4.3.3]{FrancisGaitsgory:CKD}, the analogue in that work of Theorem \ref{thm:duality}, includes the assumption that $\V$ is tensored over a field of characteristic zero. An examination of the proof reveals that this hypothesis is used only to verify that the norm maps as above are equivalences. For this reason, Theorem \ref{thm:duality} should be regarded as implicit in \cite{FrancisGaitsgory:CKD}.
\end{remark}

\subsection{Duality and factorization}\label{sec:duality and factorization}

One of the key insights of \cite{FrancisGaitsgory:CKD} is that the condition of factorizability has a particularly simple interpretation under the duality of Theorem \ref{thm:duality}. The basic calculation underlying this interpretation is the following (see the conventions for a notational reminder):

\begin{lemma}\label{graded vanish}
Suppose that $F\in\cShv_0^\cbl(\Ran_G,\C)$ is supported on the diagonal. There is a natural equivalence $$\Sym_\amalg^J(F)(U_I\subseteq M)\simeq\begin{cases}
\displaystyle\bigotimes_{I}F(U_i\subseteq M)\quad& |I|=|J|\\
0\quad&\text{else.}
\end{cases}$$
\end{lemma}
\begin{proof} We have the natural equivalence
\[\xymatrix{\Sym_\amalg^J(F)(U_I\subseteq M)\ar[r]^-\sim&\bigg(\displaystyle\bigoplus_{\Surj(I,J)}\bigotimes_JF(U_{\pi^{-1}(j)}\subseteq M)\bigg)_{\Sigma_J}.}\]
There are three cases.\begin{enumerate}
\item If $|I|<|J|$, then $\Surj(I,J)=\varnothing$, so the sum in question is empty and $\Sym_\amalg^J(F)(U_I\subseteq M)\simeq0$.
\item If $|I|=|J|$, then $\Surj(I,J)$ is a free $\Sigma_J$-set on the class of a bijection $I\cong J$, and the claim follows.
\item If $|I|>|J|$, then, for any $\pi\in \Surj(I,J)$, $|\pi^{-1}(j)|>1$ for some $j\in J$. Then $F(U_{\pi^{-1}(j)}\subseteq M)\simeq0$ by assumption, so that every term in the sum vanishes.
\end{enumerate}
\end{proof}

The proof of the following result is essentially a transcription of the argument of \cite[5.2.1]{FrancisGaitsgory:CKD}:

\begin{lemma}\label{lem:factorization}
Let $L$ be a Lie algebra in $\cShv_0^\cbl(\Ran_G,\C)$. Then $\overline C_\amalg^\Lie(L)$ is factorizable if and only if $L$ is supported on the diagonal.
\end{lemma}
\begin{proof} It suffices to prove the claim for $\gr\,\overline C_\amalg^\Lie(L)\simeq\Sym_\amalg(L[1])$ instead. We make use of the commutative diagram
$$\xymatrix{
\Sym_\amalg(L[1])\ar[r]^-\Gamma\ar[d]&\Sym_\amalg(\Sym_\amalg(L[1]))\ar[r]\ar[d]&\Sym_\amalg^{k}(\Sym_\amalg(L[1]))\ar[d]\\
\Sym_\amalg^{k}(L[1])\ar[r]&\displaystyle\bigoplus_{l\geq0}\bigg(\bigoplus_{i_1+\cdots+i_l=k}\bigotimes_{j=1}^l\Sym_\amalg^{i_j}(L[1])\bigg)_{\Sigma_l}\ar[r]^-{k=l}&\Sym_\amalg^k(L[1]),
}$$ in which the bottom composite is the identity. 

Suppose that $L$ is supported on the diagonal, and evaluate this diagram at $U_I\subseteq M$ where $|I|=k$. By Lemma \ref{graded vanish}, the outermost vertical maps become equivalences, so the top composite does as well, by two-out-of-three. Hence $\Sym_\amalg(L[1])$ is factorizable in this case.

Suppose instead that $L$ is not supported on the diagonal. Then there is some object $U_I\subseteq M$ with $|I|>1$ such that $L(U_I\subseteq M)\not\simeq0$, and we may take $|I|$ to be minimal with respect to the existence of such an object. Then $$\Sym_\amalg(L[1])(U_I\subseteq M)\simeq L(U_I\subseteq M)[1]\oplus \bigotimes_I L(U_i\subseteq M)[1]$$ by minimality, and we conclude that $\Sym_\amalg(L[1])$ is not factorizable, since $L(U_I\subseteq M)[1]\not\simeq0$.
\end{proof}

\begin{corollary}\label{cor:lie model}
Let $\C$ be a stable, presentably symmetric monoidal $\infty$-category. There is a commuting diagram \[\xymatrix{\Alg^\mathrm{nu}_{\mathbb{E}_G}(\C)\ar[d]\ar[rr] &&\Alg_\Lie(\cShv_0^\cbl(\Ran_G,\C)_\amalg)\ar[d]\\
\Mod_G(\C)\ar[r]^-{[-1]}&\Mod_G\simeq \cShv^\loc(\Diag_G,\C)\ar[r]^-{\delta_!}&\cShv_0^\cbl(\Ran_G,\C)
}\] of $\infty$-categories in which the top functor is fully faithful with essential image the subcategory of Lie algebras supported on the diagonal.
\end{corollary}
\begin{proof}
The disjoint monoidal structure on $\cShv_0^\cbl(\Ran_G,\C)$ satisfies the hypotheses of Theorem \ref{thm:duality}; indeed, the first hypothesis is Proposition \ref{prop:chiral is nilpotent}, while the second follows from the formula \[\bigotimes_J F(U_I)\simeq \bigoplus_{\Surj(I,J)}\bigotimes_J F(U_{f^{-1}(j)})\] and the observation that $\Sigma_J$ acts freely on $\Surj(I,J)$. Thus, applying Theorem \ref{thm:duality} and Lemma \ref{lem:factorization}, we obtain the indicated equivalences in the commuting diagram \[\xymatrix{
\Fact(\cShv_0^\cbl(\Ran_G,\C))\ar@{^{(}->}[d]\ar[r]^-\sim&\Alg_\Lie(\cShv_0^\cbl(\Ran_G,\C)_\amalg)\times_{\cShv_0^\cbl(\Ran_G,\C)}\cShv^\loc(\Diag_G,\C)\ar@{^{(}->}[d]\\
\Coalg_\Com^{\mathrm{nu}}(\cShv_0^\cbl(\Ran_G,\C)_\amalg)\ar[r]^-\sim&\Alg_\Lie(\cShv_0^\cbl(\Ran_G,\C)_\amalg).
}\] Composing the clockwise composite with the equivalence of Proposition \ref{prop:factorizable 2} yields the top functor in the diagram of the statement, and fully faithfulness and the identification of the essential image follow. It remains to show that this functor introduces a suspension by $-1$ at the level of underlying functors, which follows from the observation that the map $$L[1]=\overline C_\amalg^\Lie(L)_{\leq 1}\to \overline C_\amalg^\Lie(L)$$ is an equivalence when evaluated on $U_I\subseteq M$ with $|I|=1$.
\end{proof}

\subsection{Enveloping algebras}\label{sec:enveloping algebras} In this section, we prove Theorems \ref{thm:adjunction} and \ref{thm:PBW}. As indicated in the introduction, the strategy is to find a second monoidal structure on the $\infty$-category of constructible cosheaves on $\Ran_G$ whose relationship to ordinary Lie algebras parallels the relationship of the disjoint monoidal structure to $\mathbb{E}_G$-algebras. We refer to this monoidal structure as the \emph{overlapping monoidal structure} and associate to it the symbol $\cup$ (e.g. $\otimes^\cup$). The construction of the overlapping monoidal structure will require the introduction of some auxiliary concepts and is deferred to the following section (specifically, see \S\ref{section:overlapping}) . For now, we content ourselves with the following summary:

\begin{proposition}\label{prop:overlapping exists}
There is a (unital) symmetric monoidal structure on $\cShv^\cbl(\Ran_G,\C)$ in which the tensor product is given by the formula \[F\otimes^\cup G\simeq \mu_!(F\boxtimes G).\] Moreover, there is a natural transformation $\otimes^\cup\to \otimes^\amalg$ endowing the identity functor on $\cShv^\cbl(\Ran_G,\C)$ with the structure of a lax monoidal functor.
\end{proposition}

\begin{remark}
This formula should be compared to the description of the star tensor product given in \cite[1.2.1.]{FrancisGaitsgory:CKD}.
\end{remark}

\begin{proposition}\label{prop:overlapping lie}
Let $\C$ be a stable, presentably symmetric monoidal $\infty$-category. There is a commuting diagram \[\xymatrix{\Alg_\Lie(\Mod_G(\C))\ar[d]\ar[rr] &&\Alg_\Lie(\cShv^\cbl(\Ran_G,\C)_\cup)\ar[d]\\
\Mod_G(\C)\ar[r]^-{\Omega^n}&\Mod_G(\C)\simeq \cShv^\loc(\Diag_G,\C)\ar[r]^-{\delta_!}&\cShv^\cbl(\Ran_G,\C)
}\] of $\infty$-categories in which the top functor is fully faithful with essential image the subcategory of Lie algebras supported on the diagonal.
\end{proposition}
\begin{proof}
By Lemma \ref{lem:base change}, the left adjoint in the $(\delta^*,\delta_!)$-adjunction is symmetric monoidal with respect to the monoidal structure on the domain given by $\otimes^\cup$ and the monoidal structure on the codomain given by $\otimes^*$; therefore, the adjunction lifts to an adjunction at the level of $\infty$-categories of Lie algebras. The result now follows from Lemma \ref{lem:verdier} and Corollary \ref{cor:diagonal support}.
\end{proof}

\begin{proof}[Proof of Theorem \ref{thm:adjunction}]
The identity functor preserves the condition of diagonal support; therefore, we obtain a forgetful functor as the dashed factorization in the diagram \[\xymatrix{
\Alg^{\mathrm{nu}}_{\mathbb{E}_G}(\C)\ar@{-->}[dd]\ar@{^{(}->}[rr]^-{(\ref{cor:lie model})} &&\Alg_\Lie(\cShv^\cbl(\Ran_G,\C)_\amalg)\ar[dd]^-{\id}\\\\
\Alg_{\Lie}(\Mod_G(\C))\ar@{^{(}->}[rr]^-{(\ref{prop:overlapping lie})}&&\Alg_\Lie(\cShv^\cbl(\Ran_G,\C)_\cup)
}\] Since each arrow in the diagram preserves limits and filtered colimits, which in each case are underlying, this forgetful functor admits a left adjoint $U_G$. To complete the proof, it suffices to note that the diagram of right adjoints \[\xymatrix{
\Alg^\mathrm{nu}_{\mathbb{E}_G}(\C)\ar[d]\ar[r]&\Alg_{\Lie}(\Mod_G(\C))\ar[d]\\
\Mod_G(\C)\ar[r]^-{\Sigma^n[-1]}&\Mod_G(\C).
}\] commutes.
\end{proof}

Having exhibited the desired forgetful functor, we turn now to the task of making its left adjoint $U_G$ explicit. Since the identity preserves limits and filtered colimits of Lie algebras, it admits a left adjoint, which we denote by $\Ind_\cup^\amalg$. The key result concerning this functor is the following.

\begin{lemma}\label{envelope support}
If $L$ is supported on the diagonal, then so is $\Ind_\cup^\amalg(L)$.
\end{lemma}

The proof, essentially a transcription of \cite[6.4.2]{FrancisGaitsgory:CKD}, is premised on the following calculation in the overlapping monoidal structure:

\begin{lemma}\label{star sym}
Suppose that $F\in\cShv^\cbl(\Ran_G,\C)$ is supported on the diagonal. There is a natural equivalence $$\xymatrix{\Sym_\cup^{k}(F)(U_I\subseteq M)\ar[r]^-\sim&\displaystyle\bigoplus_{\sum_Ik_i=k}\bigotimes_{I}\Sym^{k_i}(F(U_i\subseteq M)).}$$ 
\end{lemma}
\begin{proof}
We have the $\Sigma_k$-equivariant equivalence $$\xymatrix{F^{\otimes^\cup k}(U_I\subseteq M)\ar[r]^-\sim&\displaystyle\bigoplus_{\Cov(I, k)}\bigotimes_{j=1}^kF(U_{S_j}\subseteq M)}.$$ Since $F$ is supported on the diagonal, the summand corresponding to a $k$-cover $S$ vanishes if $|S_j|\neq 1$ for any $1\leq j\leq k$. But the set of $k$-covers $S$ of $I$ with $|S_j|=1$ for all $j$ is put in $\Sigma_k$-equivariant bijection with $\Surj(\{1,\ldots,k\},I)$ by sending $j$ to the unique element in $S_j$, so we may write $$\xymatrix{F^{\otimes^\cup k}(U_I\subseteq M)\ar[r]^-\sim&\displaystyle\bigoplus_{\Surj(\{1,\ldots, k\},I)}\bigotimes_IF(U_i\subseteq M)^{\otimes \pi^{-1}(i)}}=\bigoplus_{\sum_Ik_i=k}\Ind_{\prod_I\Sigma_{k_i}}^{\Sigma_k}\bigg(\bigotimes_IF(U_i\subseteq M)^{\otimes k_i}\bigg).$$ Passing to $\Sigma_k$-coinvariants yields the claim.
\end{proof}

We also make use of the following result concerning the interaction of Lie chains with the induction functor, which is an immediate corollary of \cite[6.2.6]{FrancisGaitsgory:CKD}.

\begin{lemma}\label{lem:envelope is star}
The diagram $$\xymatrix{
\Alg_\Lie(\cShv^\cbl(\Ran_G,\C)_\cup)\ar[d]_-{ C_\cup^\Lie}\ar[r]^-{\Ind_\cup^\amalg}&\Alg_\Lie(\cShv^\cbl(\Ran_G,\C)_\amalg)\ar[d]^-{ C^\Lie_\amalg}\\
\Coalg_{\Com}(\cShv^\cbl(\Ran_G,\C)_\cup)\ar[r]^-{\id}&\Coalg_{\Com}(\cShv^\cbl(\Ran_G,\C)_\amalg)}$$ commutes.
\end{lemma}

\begin{proof}[Proof of Proposition \ref{envelope support}]
By Lemma \ref{lem:factorization}, it suffices to show that $C_\amalg^\Lie(\Ind_\cup^\amalg(L))$ is factorizable. By Lemma \ref{lem:envelope is star}, we have $C_\amalg^\Lie(\Ind_\cup^\amalg(L))\simeq C_\cup^\Lie(L)$; therefore, passing to the graded coalgebra associated to the filtration of Theorem \ref{thm:duality}, applied in the overlapping monoidal structure, it will suffice to show that $\Sym_\cup(L[1])$ is factorizable when $L$ is supported on the diagonal. But by Lemma \ref{star sym}, \begin{align*}\Sym_\cup(L[1])(U_I\subseteq M)&\simeq \bigoplus_{k\geq1}\bigoplus_{\sum_Ik_i=k}\bigotimes_{I}\Sym^{k_i}(L(U_i\subseteq M)[1])\\&\simeq \bigotimes_I\Sym(L(U_i\subseteq M)[1])\\&\simeq \bigotimes_I\Sym_\cup(L[1])(U_i\subseteq M),
\end{align*} as desired.
\end{proof}

\begin{corollary}\label{cor:adjoint identification}
For a Lie algebra $L$ in $\Mod_G(\C)$, $U_G(L)$ is the nonunital $\mathbb{E}_G$-algebra corresponding to $C_\amalg^\Lie(\Ind_\cup^\amalg(\delta_!L_c)).$
\end{corollary}

Now, Lemma \ref{lem:preserves factorizability} endows $\pi_!\simeq \jmath_!\epsilon_!$ with a symmetric monoidal structure preserving factorizability, and we have the following commuting diagram:

\[
\xymatrix{
\Alg^\mathrm{nu}_{\mathbb{E}_G}(\C)\ar[d]_-{\wr}\ar[r]^-{\int}& \Fun^\otimes(\Mfld_G,\C)\ar@{^{(}->}[dd]\\
\Fact(\cShv^\cbl_0(\Ran_G,\C))\ar@{^{(}->}[d]_-{(-)\oplus 1}\\
\Coalg_\Com(\cShv^\cbl(\Ran_G,\C)_\amalg)\ar[r]^-{\pi_!}&\Coalg_\Com(\Fun(\Mfld_G,\C)_\mathrm{RD})
}
\]

\begin{bigtheorem}\label{thm:PBW}
Let $L$ be a Lie algebra in $\Mod_G(\C)$. There is a natural equivalence of augmented $\mathbb{E}_G$-algebras \[U_G(L)\oplus 1_\C\simeq C^\Lie(\Omega^n L).\]
\end{bigtheorem}
\begin{proof}
We have the sequence of equivalences
\begin{align*}\int U_G(L)&\simeq \pi_!C^\Lie_\amalg(\Ind_\cup^\amalg(\delta_!L_c))\quad&(\ref{cor:adjoint identification})\\
&\simeq \pi_!C^\Lie_\cup(\delta_!L_c)\quad&(\ref{lem:envelope is star})\\
&\simeq C^\Lie(\pi_!\delta_! L_c).\quad&(\ref{lem:globalization chains})
\end{align*}

Restricting to $\D_G$ and invoking Lemma \ref{lem:verdier}, we obtain the desired equivalence \[U_G(L)\oplus 1_\C\simeq C^\Lie(\Omega^nL).\]
\end{proof}

\subsection{Application to configuration spaces}\label{sec:configurations section}

We now take up the proof of Theorem \ref{thm:configurations}. For this, we specialize to the case $G=\Top(n)$ and take $\C=\Fun(\Sigma,\mathcal{S}\mathrm{p})$ to be the $\infty$-category of symmetric sequences in spectra equipped with the (left) Day convolution monoidal structure \[(X\otimes Y)_k\simeq \bigoplus_{i+j=k}\Sigma_k\wedge_{\Sigma_i\times\Sigma_j}X_i\wedge Y_j.\] We consider spectra as symmetric sequences concentrated in arity 1 and write $\mathbb{S}$ for the sphere spectrum.

We will be interested in the symmetric sequence $\Sigma_+^\infty\Conf_\bullet(M)$ of ordered configuration spaces of $M$. This object may be interpreted in terms of factorization homology:

\begin{lemma}\label{lem:free}
There is a natural equivalence \[\Sigma_+^\infty\Conf_\bullet(M)\simeq \int_M\mathbb{E}_G^\mathrm{nu}(\mathbb{S})\]
\end{lemma}
\begin{proof}
The claim is immediate from the argument of \cite[5.5]{AyalaFrancis:FHTM} and the easy calculation that $\mathbb{S}^{\otimes k}$ is $\Sigma_k\wedge \mathbb{S}$, thought of as a symmetric sequence concentrated in arity $k$.
\end{proof}

\begin{proof}[Proof of Theorem \ref{thm:configurations}]
We have the natural equivalences
\begin{align*}
\Sigma_+^\infty\Conf_\bullet(M)&\simeq\int_M\mathbb{E}^\mathrm{nu}_G(\mathbb{S})&\qquad (\ref{lem:free}) \\
&\simeq \int_MU_G(\Lie(\Sigma^n\mathbb{S}^{-1}))&\quad (\ref{thm:adjunction})\\
&\simeq \int_M C^\Lie(\Omega^n\Lie(\Sigma^n\mathbb{S}^{-1}))&\quad (\ref{thm:PBW})\\
&\simeq C^\Lie(\Map^{\Top(n)}(\mathrm{Fr}_{M^+}, \Lie(\Sigma^n \mathbb{S}^{-1})))&\quad \text{\cite[5.13]{AyalaFrancis:FHTM}},
\end{align*} where $\Fr$ denotes the $\Top(n)$-bundle associated to the microtangent bundle.

Thus, it will suffice to show that this spectral Lie algebra depends only on the homotopy type of $M^+$. For this, we observe that the action of $\Top(n)$ on $\Sigma^n\mathbb{S}^{-1}$ factors through the stable $J$-homomorphism \[\Top(n)\to \mathrm{hAut}(\mathbb{S}),\] so that, for any $N\geq n$, we have \[\Map^{\Top(n)}(\mathrm{Fr}_{M^+}, \Lie(\Sigma^n \mathbb{S}^{-1}))\simeq \Map^{\mathrm{hAut}(S^N)}(\Fr_{M^+}\times_{\Top(n)}\mathrm{hAut}(S^N), \Lie(\mathbb{S}^{n-1})),\] and the claim follows, since the stable spherical fibration associated to the tangent microbundle is a homotopy invariant by Atiyah duality (see \cite{Atiyah:TC} in the smooth case).
\end{proof}

\begin{remark}
In unpublished work, John Francis gives an alternate proof of this result using Goodwillie calculus.
\end{remark}

\section{Functoriality}\label{sec:functoriality section}In this section, we define the pushforward and pullback functors discussed in \S\ref{sec:push pull section}, and we study their interactions.

\subsection{The main diagonal}\label{sec:main diagonal} We investigate functoriality arising from $\delta:\Diag_G\to \Ran_G$, beginning with the left Kan extension $\delta_!$. From the definitions, it is clear that the only arrows in $\Ran_G$ with source lying in $\Diag_G$ also have target lying in $\Diag_G$, so we have the following calculation:

\begin{lemma}\label{lem:diagonal push formula 1}
Let $F:\Diag_G\to \C$ be any functor. There is a natural equivalence \[
\delta_!F(U_I\subseteq M)\simeq \begin{cases}
F(U_I\subseteq M)&\qquad |I|=1\\
0&\qquad \text{else.}
\end{cases}
\] In particular, if $F$ is a locally constant cosheaf, then $\delta_!F$ is a reduced constructible cosheaf.
\end{lemma}

\begin{corollary}\label{cor:diagonal support}
The pushforward $\delta_!:\cShv^\loc(\Diag_G,\C)\to \cShv^\cbl(\Ran_G,\C)$ is fully faithful with essential image the subcategory of constructible cosheaves supported on the diagonal.
\end{corollary}

Combining Lemma \ref{lem:diagonal push formula 1} with the observation that the restriction of a constructible cosheaf is locally constant, we obtain an adjunction \[\delta_!:\cShv^\loc(\Diag_G,\C)\rightleftarrows \cShv^\cbl(\Ran_G,\C):\delta^!,\] where $\delta^!$ denotes the restriction of $\delta^\natural$ to the indicated domain and codomain. Since limits and colimits on both sides are pointwise, Lemma \ref{lem:diagonal push formula 1} implies that $\delta_!$ preserves limits and therefore also admits a left adjoint $\delta^*$.

Although we will not need it, a formula for the exceptional pullback $\delta^*$ is easily obtained by imitating classical arguments (see \cite[3]{Milne:EC}, for example). Let $j$ denote the inclusion into $\Ran_G$ of the full sucategory of objects $U_I\subseteq M$ with $|I|>2$.

\begin{lemma}\label{lem:exceptional formula}
There is a natural equivalence \[\delta^*F\simeq\mathrm{cofib}\big(\delta^!j_!j^\natural F\to \delta^!F\big),\] where $j$ denotes the inclusion of the full sucategory of objects $U_I\subseteq M$ with $|I|>2$. 
\end{lemma}

\subsection{Matched disks}\label{sec:matched disks} In order to study the analogue of the multiplication $\Ran(M)\times\Ran(M)\to\Ran(M)$ in the absolute context, we require a preliminary definition. 

\begin{definition}
An object $(U_I\subseteq M, V_J\subseteq M)\in \Ran_G\times_{\Mfld_G} \Ran_G$ is \emph{matched} if $U_i\cap V_j=\varnothing$ whenever $U_i\neq V_j$.
\end{definition}

We write $(\Ran_G\times_{\Mfld_G}\Ran_G)_\#$ for the full subcategory of $\Ran_G\times_{\Mfld_G}\Ran_G$ spanned by the matched objects. Note that, if $(U_I\subseteq M,V_J\subseteq M)$ is a matched pair, then $U_I\cup V_J\subseteq M$ is again an object of $\Ran_G$, which we denote by $\mu(U_I\subseteq M, V_J\subseteq M)$. This construction extends in an obvious way to yield a functor \[\mu:(\Ran_G\times_{\Mfld_G}\Ran_G)_\#\to \Ran_G.\]

In order to formulate the corresponding push and pull functors, we make the following obvious definition:

\begin{definition}
Let $F:\Ran_G\times_{\Mfld_G}\Ran_G\to \C$ be a functor. 
\begin{enumerate}
\item We say that $F$ is a \emph{constructible cosheaf} if $F$ sends pairs of isotopy equivalences in to equivalences in $\C$.
\item We say that $F$ is \emph{reduced} if $F(\varnothing\subseteq \varnothing, -)\simeq F(-,\varnothing\subseteq \varnothing)\simeq 0$. 
\end{enumerate}
\end{definition}

We defer the proof of the next result to the end of the section for the sake of continuity.

\begin{lemma}\label{lem:matched extension}
If $F:\Ran_G\times_{\Mfld_G}\Ran_G\to \C$ is a constructible cosheaf, then $F$ is the left Kan extension of its restriction to $(\Ran_G\times_{\Mfld_G}\Ran_G)_\#$.
\end{lemma}

Thus, constructible cosheaves on the square of $\Ran_G$ may be identified with a full subcategory of functors defined on the smaller category of matched objects. Using this observation, we are able to define the desired pullback $\mu^!$ as the dashed factorization in the diagram

\[\xymatrix{
\cShv^\cbl(\Ran_G,\C)\ar@{-->}[dd]_-{\mu^!}\ar@{^{(}->}[r]& \Fun(\Ran_G,\C)\ar[dd]^-{\mu^\natural}\\\\
\cShv^\cbl(\Ran_G\times_{\Mfld_G}\Ran_G,\C)\ar@{^{(}->}[r]& \Fun((\Ran_G\times_{\Mfld_G}\Ran_G)_\#,\C).
}\] We remark that $\mu^!$ preserves the condition of being reduced.

By Kan extension, the restriction $\mu^\natural$ admits left and right adjoints. We have the following calculation concerning these functors:

\begin{lemma}\label{lem:multiplication push formula}
There are natural equivalences \[\mu_!F(U_I\subseteq M)\simeq \bigoplus_{I_1\cup I_2=I}F(U_{I_1}\subseteq M, U_{I_2}\subseteq M)\simeq \mu_*F(U_I\subseteq M).\] In particular, if $F$ is a (reduced) constructible cosheaf, then so are $\mu_!F$ and $\mu_*F$
\end{lemma}
\begin{proof}
Since $\mu^{-1}(U_I\subseteq M)=\{(U_{I_1}\subseteq M, U_{I_2}\subseteq M)\}_{I_1\cup I_2=I}$ and $\C$ is stable, it suffices to note that the obvious functors $\mu^{-1}(U_I\subseteq M)\to \mu_{/U_I\subseteq M}$ and $\mu^{-1}(U_I\subseteq M)\to \mu_{U_I\subseteq M/}$ are final and initial, respectively.
\end{proof}

From the formula of Lemma \ref{lem:multiplication push formula}, we see that the two pushforwards do in fact coincide, a fact corresponding to the properness of the multiplication map $\Ran(M)^2\to \Ran(M)$. Indeed, by imitating classical arguments involving the relative diagonal, one can produce the natural transformation $\mu_!\to \mu_*$ inducing this equivalence.

We close with an examination of the interaction between $\mu_!$ and the globalization functor $\pi_!$ of \S\ref{sec:Ran section}. We denote by $\gamma:\Mfld_G\to \Mfld_G\times\Mfld_G$ the diagonal functor.

\begin{lemma}\label{lem:global sections}
Let $F$ be a constructible cosheaf on $\Ran_G\times_{\Mfld_G}\Ran_G$. There is a natural equivalence \[\pi_!\mu_!F\xrightarrow{\sim}\gamma^\natural(\pi\times\pi)_!F\] in $\Fun^\loc(\Mfld_G,\C)$.
\end{lemma}
\begin{proof}
The obvious equality $\gamma\pi\mu\iota=(\pi\times\pi)\iota$ gives rise to the map. To show that this map is an equivalence, it will suffice to note that it is so after evaluating at a manifold $M$, where both sides become $\colim_{\D(M)^2}F$.
\end{proof}

\subsection{Base change}\label{sec:base change} The goal of this section is to prove an analogue of the classical proper base change theorem applied to the pullback square \[\xymatrix{
M\ar@{=}[d]\ar[r]&M\times M\ar[r]&\Ran(M)\times\Ran(M)\ar[d]\\
M\ar[rr]&&\Ran(M).
}\] In order to formulate this result, we must first introduce the analogue of the upper left horizontal map, the diagonal of $M$.

We write $(\Diag_G\times_{\Mfld_G}\Diag_G)_\#\subseteq \Diag_G\times_{\Mfld_G}\Diag_G$ for the full subcategory of matched objects and denote by $\Delta:\Diag_G\to (\Diag_G\times_{\Mfld_G}\Diag_G)_\#$ the relative diagonal functor \[\Delta(U\subseteq M)=\Delta(U\subseteq M,U\subseteq M).\] As in Lemma \ref{lem:matched extension}, a locally constant cosheaf on $\Diag_G\times_{\Mfld_G}\Diag_G$ is determined by its restriction to $(\Diag_G\times_{\Mfld_G}\Diag_G)_\#$, so we obtain a pullback functor $\Delta^!$ as the dashed factorization in the diagram 

\[\xymatrix{
\cShv^\loc(\Diag_G,\C)\ar@{^{(}->}[r]& \Fun(\Diag_G,\C)\\\\
\cShv^\loc(\Diag_G\times_{\Mfld_G}\Diag_G,\C)\ar@{-->}[uu]^-{\Delta^!}\ar@{^{(}->}[r]& \Fun((\Diag_G\times_{\Mfld_G}\Diag_G)_\#,\C)\ar[uu]_-{\Delta^\natural}.
}\]

Regarding the left adjoint $\Delta_!$, we have the following easy calculation:

\begin{lemma}\label{lem:diagonal push formula 2}
Let $F:\Diag_G\to \C$ be any functor. There is a natural equivalence \[\Delta_!F(U\subseteq M, V\subseteq M)\simeq\begin{cases}
F(U\subseteq M)&\quad U=V\\
0&\quad \text{else.}
\end{cases}\] In particular, if $F$ is a locally constant cosheaf, then so is $\Delta_!F$.
\end{lemma}

As with $\delta$, we see that $\Delta_!$ also admits a left adjoint $\Delta^*$. A similar formula to that of Lemma \ref{lem:exceptional formula}, proved in the same way, holds for $\Delta^*$.

Returning to the main thread, the pullback square exhibited above corresponds in the absolute context to the diagram

\[\xymatrix{
\Diag_G\ar@{=}[dd]\ar[r]^-\Delta&(\Diag_G\times_{\Mfld_G} \Diag_G)_\#\ar[r]^-{\delta\times\delta}&(\Ran_G\times_{\Mfld_G}\Ran_G)_\#\ar[dd]^-\mu\\\\
\Diag_G\ar[rr]^-\delta&&\Ran_G.
}\] Since this diagram commutes, there is an induced base change morphism (see \cite[7.3.1]{Lurie:HTT}), and we have the following:

\begin{lemma}\label{lem:base change}
The canonical map $\delta^*\mu_!\to \Delta^*(\delta\times\delta)^*$ is an equivalence.
\end{lemma}
\begin{proof}
Passing to right adjoints, it suffices instead to show the equivalence $(\delta\times\delta)_!\Delta_!\simeq\mu^!\delta_!$. For this, we note that \begin{align*}(\delta\times\delta)_!\Delta_!F(U_{I}\subseteq M, V_{J}\subseteq M)&\overset{(\ref{lem:diagonal push formula 1})}{\simeq}\begin{cases}
\Delta_!F(U_I\subseteq M, V_J\subseteq M) &\quad |I|=|J|=1\\
0&\quad \text{else}
\end{cases}\\
&\overset{(\ref{lem:diagonal push formula 2})}{\simeq}\begin{cases}
F(U_I\subseteq M)&\quad U_I=V_J,\, |I|=|J|=1\\
0&\quad\text{else.}
\end{cases}
\end{align*}
On the other hand, since $\mu(U_I\subseteq M, V_J\subseteq M)\in \Diag_G$ if and only if $U_I=V_J$ and $|I|=|J|=1$, we likewise have $$\mu^!\delta_!F(U_I\subseteq M, V_J\subseteq M)\simeq\delta_!F(\mu(U_I\subseteq M, V_J\subseteq M))\simeq\begin{cases}
F(U_I\subseteq M)&\quad U_I=V_J,\, |I|=|J|=1\\
0&\quad\text{else.}
\end{cases}$$
\end{proof}

\subsection{Proof of Lemma \ref{lem:matched extension}}

Rather than a direct proof making reference only to the categories in question, we choose to offer a topological proof illustrating the close relationship between the category $\Ran_G$ and the spaces $\Ran(M)$.

Denoting by $\D(M)^2_\#$ the pullback in the diagram \[\xymatrix{
\D(M)^2_\#\ar[d]\ar[r]&\D(M)^2\ar[d]\\
(\Ran_G\times_{\Mfld_G}\Ran_G)_\#\ar[r]&\Ran_G\times_{\Mfld_G}\Ran_G,
}\] we have the following commuting diagram of functors:

\[\xymatrix{
(\Ran_G\times_{\Mfld_G}\Ran_G)_\#\ar[rr]^-{f}&&\Ran_G\times_{\Mfld_G}\Ran_G\\
\D(M)^2_\#\ar[u]^-{i_M^2}\ar[dr]\ar[rr]^-{f_M}&&\D(M)^2\ar[dl]^-j\ar[u]_-{i_M^2}\\
&\mathrm{Op}(\Ran(M)^2).
}\]

The topological input to the proof is the following:

\begin{lemma}\label{lem:basis}
The image of $\D(M)^2_\#$ in $\mathrm{Op}(\Ran(M)^2)$ is a basis for the topology of $\Ran(M)^2$.
\end{lemma}
\begin{proof}
By definition, the topology of $\Ran(M)^2$ is generated by the image of $\D(M)^2$. We will show that the topology generated by the image of $\D(M)_\#^2$ is finer than this topology; since the converse obviously holds, the proof will be complete.

Fix $(U_I, V_J)\in\D(M)^2$ and finite subsets $S\subseteq U_I$ and $T\subseteq V_J$ whose inclusions surject on connected components. For each $x\in S\cup T$, we choose a Euclidean neighborhood $x\in W_x\subseteq M$ such that \[W_x\subseteq \bigg(\bigcap_{\{i\in I,\,x\in U_i\}}U_i\bigg)\cap\bigg(\bigcap_{\{j\in J,\,x\in V_j\}}V_j\bigg)\] without loss of generality, we may assume that $W_x\cap W_y=\varnothing$ for $x\neq y$. Then 
\begin{enumerate}
\item $(W_S, W_T)\in\D(M)^2_\#$,
\item $S\subseteq W_S\subseteq U_I$,
\item $T\subseteq W_T\subseteq V_J$, and
\item the inclusion $(W_S, W_T)\subseteq (U_I, V_J)$ lies in $\D(M)^2$. 
\end{enumerate} Thus, the topology generated by the image of $\D(M)^2_\#$ is finer than the topology generated by $\D(M)^2$.
\end{proof}

We will deduce the global result from the following relative version:

\begin{lemma}\label{lem:local counit}
If $F:\D(M)\to \C$ is locally constant, then $F$ is the left Kan extension of its restriction to $\D(M)^2_\#$.
\end{lemma}
\begin{proof}
Since $j_!F$ is a cosheaf on $\Ran(M)$, and since a cosheaf is left Kan extended from any choice of basis, Lemma \ref{lem:basis} implies that the counit $j_!(f_M)_!f_M^\natural j^\natural j_! F\to j_!F$ is an equivalence. Using the fact that the inclusion $j$ is fully faithful twice, we obtain the desired equivalence \[(f_M)_!f^\natural F\simeq (f_M)_!f_M^\natural j^\natural j_!F\simeq F.\]
\end{proof}

\begin{proof}[Proof of Lemma \ref{lem:matched extension}]
Let $F:\Ran_G\times_{\Mfld_G}\Ran_G\to \C$ be a constructible cosheaf. We wish to show that the counit $f_!f^\natural F\to F$ is an equivalence.

Now, fix an object $(U_I\subseteq M, V_J\subseteq M)$. If either $I$ or $J$ is empty, then $f_{(U_I\subseteq M, V_J\subseteq M)}=(\Ran_G\times_{\Mfld_G}\Ran_G)_{/(U_I\subseteq M, V_J\subseteq M)}$ and there is nothing to prove, so assume otherwise. Then the evident functor $f_{M/(U_I, V_J)}\to f_{/(U_I\subseteq M, V_J\subseteq M)}$ is final, and, using Lemma \ref{lem:local counit}, we have \begin{align*}f_!f^\natural F(U_I\subseteq M, V_J\subseteq M)&\simeq (f_M)_!(i_M^2)^\natural f^\natural F(U_I, V_J)\\
&\simeq (f_M)_!f_M^\natural (i_M^2)^\natural F(U_I, V_J)\\
&\simeq (i_M^2)^\natural F(U_I, V_J)\\
&\simeq F(U_I\subseteq M, V_J\subseteq M).
\end{align*}
\end{proof}

\section{Monoidal structures}\label{sec:monoidal section}In this section, we provide proofs of Propositions \ref{prop:disjoint day} and \ref{prop:overlapping exists}. These arguments pass through the formalism of Day convolution, which we begin by reviewing.

\subsection{Day convolution}\label{sec:day convolution} We rely heavily on the following result of \cite{Glasman:DCIC} (see also \cite[4.8.1]{Lurie:HA} for a version with $\W$ the $\infty$-category of spaces and \cite{Day:CCF} for the original 1-categorical version).

\begin{theorem}[Glasman]\label{thm:glasman}
Let $\V$ and $\W$ be symmetric monoidal $\infty$-categories, and assume that \begin{enumerate}
\item $\W$ has colimits, and
\item the tensor product of $\W$ distributes over colimits indexed by the $\infty$-categories $\otimes_{/v}$ for $v\in \V$, where $\otimes:\V\times\V\to \V$ denotes the tensor product functor of $\V$.
\end{enumerate} There is a symmetric monoidal structure on $\Fun(\V,\W)$, called \emph{left Day convolution}, in which the tensor product is given by the formula \[(F\otimes^{\mathrm{LD}} G)(v)\simeq \colim\bigg(\otimes_{/v}\to \V\times\V\xrightarrow{F\times G} \W\times\W\xrightarrow{\otimes}\W\bigg).\] Moreover, there is an equivalence \[\Fun^{\lax}(\V,\W)\simeq \Alg_\Com(\Fun(\V,\W))\] covering the identity on $\Fun(\V,\W)$. 
\end{theorem}
\begin{proof}
We explain how each of these claims is either explicit or implicit in \cite{Glasman:DCIC}. For the duration of this explanation, we make free use of terminology and notation pertaining to the theory of symmetric monoidal $\infty$-categories developed in \cite[2]{Lurie:HA}. 

The symmetric monoidal structure in question, whose definition appears as \cite[2.8]{Glasman:DCIC}, is exhibited there as a subobject of a larger locally coCartesian fibration $\overline{\Fun(\V,\W)^\otimes}$, which is simply the internal mapping object of maps between $\V^\otimes$ and $\W^\otimes$ in the category of simplicial sets over the nerve of the category of pointed finite sets. The heart of the proof that $\Fun(\V,\W)$ does in fact inherit a symmetric monoidal structure from this larger object is \cite[2.10]{Glasman:DCIC}. This argument is carried out under the assumption that the tensor product of $\W$ distributes over \emph{all} colimits, but an examination of the proof reveals that the only types of colimits that appear are of the form covered by our assumption (2). Thus, the argument given there applies without change. 

Finally, the formula for the tensor product follows from the description of the locally coCartesian edges in $\overline{\Fun(\V,\W)^\otimes}$ given in \cite[2.4]{Glasman:DCIC}, and the equivalence between lax monoidal functors and commutative algebras is \cite[2.12]{Glasman:DCIC}.
\end{proof}

Under this equivalence, the lax structure maps $F(v_1)\otimes F(v_2)\to F(v_1\otimes v_2)$ of a lax monoidal functor $F$ are identified with the components of the commutative multiplication on $F$. In particular, we obtain an identification of $\Fun^\otimes(\V,\W)$ with the full subcategory of commutative algebra objects having the property that these components are equivalences in $\W$. 

The monoidal structure of Theorem \ref{thm:glasman} is traditionally called simply Day convolution. For reasons that will become apparent presently, we prefer the name \emph{left} Day convolution, a choice which is justified by the observation that the formula for the convolution of functors $F$ and $G$ is nothing other than the left Kan extension in the diagram

\[\xymatrix{\V\times\V\ar[d]_-{\otimes_\V}\ar[r]^-{F\times G}& \W\times\W\ar[r]^-{\otimes_\W}&\W\\
\V
}\]

We wish to contemplate taking the \emph{right} Kan extension in the same diagram. Unwinding the dualities, we arrive at the following definition:

\begin{definition}
Let $\V$ and $\W$ be symmetric monoidal $\infty$-categories. \emph{Right Day convolution}, if it exists, is the symmetric monoidal structure opposite to left Day convolution on $\Fun(\V^{op}, \W^{op})$. 
\end{definition}

In the presence of right Day convolution on $\Fun(\V,\W)\simeq\Fun(\V^{op},\W^{op})^{op}$, we obtain the identification \[\Fun^{\oplax}(\V,\W)\simeq \Fun^\lax(\V^{op},\W^{op})^{op}\simeq \Alg_\Com(\Fun(\V^{op},\W^{op}))^{op}\simeq \Coalg_\Com(\Fun(\V,\W)).\] In particular, $\Fun^\otimes(\V,\W)$ is identified with the full subcategory of cocommutative coalgebra objects for which the relevant components of the comultiplication are equivalences in $\W$. 

It should be emphasized that, while the conditions guaranteeing the existence of left Day convolution are fairly innocuous---in many cases one is interested in a target whose tensor product distributes over \emph{all} colimits---the dual conditions are quite restrictive. Nevertheless, we have the following existence criterion, which, although not maximally general, will suffice for our purposes:

\begin{corollary}\label{cor:right day}
Let $\V$ and $\W$ be symmetric monoidal $\infty$-categories, and assume that \begin{enumerate}
\item $\W$ is stable and presentably symmetric monoidal, and
\item for each $v\in \V$, the undercategory $\otimes_{v/}$ receives an initial functor from a finite $\infty$-category.
\end{enumerate}
Then right Day convolution exists on $\Fun(\V,\W)$.
\end{corollary}
\begin{proof}
Since $\W$ is stable and the tensor product of $\W$ distributes over colimits, and in particular finite colimits, it also distributes over finite limits. Applying (2), the result follows from Theorem \ref{thm:glasman}.
\end{proof}

We shall denote the tensor product of right Day convolution by $\otimes^\mathrm{RD}$.

We close this section by applying the ideas introduced in the previous section to the situation in which the domain symmetric monoidal $\infty$-category is $\D_G$. To this end, we consider the undercategory $(\amalg_J)_{U_I/}$, and we note that a map $f:I\to J$ of finite sets determines an object of this undercategory, namely the canonical isomorphism $\varphi_f:U_I\cong \amalg_J U_{f^{-1}(j)}$. We obtain in this way a functor $\varphi_{(-)}$ from the set $\Fin(I,J)$, viewed as a discrete category, and we have the following easy result about this functor:

\begin{lemma}\label{finite}
For any $U_I\in\D_G$, the functor $\varphi_{(-)}:\Fin(I,J)\to (\amalg_J)_{U_I/}$ is initial.
\end{lemma}

Combining this calculation with Corollary \ref{cor:right day}, we obtain Proposition \ref{prop:disjoint day} as an immediate corollary.

\subsection{Overlapping disks}\label{sec:overlapping disks} Our construction of the overlapping monoidal structure will be premised on the relationship between $\Ran_G$ and a second combinatorial structure emerging from collections of disks. In this section we introduce the requisite preliminaries, and we construct the overlapping monoidal structure in the next.

\begin{definition} Let $I$, $J$ and $K$ be finite sets. 
\begin{enumerate}
\item A $J$-\emph{cover} of $I$ is a $J$-indexed collection $S=S_{J}$ of subsets of $I$ such that $\bigcup_{J} S_j=I$.
\item If $S$ is a $J$-cover of $I$ and $T$ is a $K$-cover of $J$, the \emph{composite} $T\circ S$ is the $K$-cover of $I$ defined by $(T\circ S)_k=\cup_{j\in T_k}S_j$.
\item Let $S$ be a $J$-cover of $I$ and $T$ an $L$-cover of $K$. The \emph{disjoint union} of $S$ and $T$ is the $J\amalg L$ cover of $I\amalg K$ given by $$(S\amalg T)_r=\begin{cases}
S_r&\quad r\in J\\
T_r&\quad r\in L.
\end{cases}$$
\end{enumerate}
\end{definition}

Covers form a category $\Cov$ with objects finite sets, morphisms from $I$ to $J$ the set $\Cov(I,J)$ of $J$-covers of $I$, and composition defined by composition of covers. We view $\Cov$ as a symmetric monoidal category under disjoint union of sets and covers. A function $f:I\to J$ determines a $J$-cover $S(f)$ of $I$ given by $S(f)_j=f^{-1}(j)$, and this assignment extends to a symmetric monoidal functor $\Fin\to \Cov$.

A cover may be pictured graphically as a system of lines drawn between elements of $I$ and elements of $J$, where a line connects $i$ and $j$ if and only if $i\in S_j$. Such a system of lines determines a cover precisely when every element of $i$ is connected to some element of $j$, and the cover is a function precisely when each $i$ is connected to exactly one $j$. At the other extreme from functions are what one might think of as ``splittings,'' where each $j$ is connected to exactly one $i$. As the following example shows, the addition of these splitting covers is the essential difference between $\Fin$ and $\Cov$.

\begin{example}\label{canonical cover}
A $J$-cover $S$ of $I$ determines a canonical cover $\widetilde S:I\to \amalg_J S_j$ by $(\widetilde S)_i=\{i\}$, and $S$ factors uniquely as $\widetilde S$ followed by the obvious function $\amalg_J S_j\to J$.
\end{example}

\begin{definition}\label{overlapping disks}
The \emph{category of overlapping ($G$-framed) disks} is the category $\widetilde\D_G$ specified by the following data:
\begin{enumerate}
\item the objects of $\widetilde\D_G$ are the objects of $\D_G$;
\item a morphism from $U_I$ to $V_J$ in $\widetilde\D_G$ is a $J$-cover $S$ of $I$ and an element of $\prod_J\Hom_{\D_G}(U_{S_j}, V_j)$;
\item composition is given by composition of covers and composition in $\D_G$.
\end{enumerate}
\end{definition}

\begin{example}\label{canonical morphism}
Given an object $U_I$ and a $J$-cover of $I$, there is a canonical morphism $\varphi_S:U_I\to \amalg_{J}U_{S_j}$ lying over the cover $\widetilde S$ of Example \ref{canonical cover}, the components of which are the identites of the various $U_i$. The presence of these morphisms, which allow disks to ``split apart,'' is the essential difference between $\D_G$ and $\widetilde\D_G$.
\end{example}

The category of overlapping disks is symmetric monoidal under disjoint union and equipped with a symmetric monoidal functor to $\Fin$.

We now imitate our earlier our approach to $\D_G$; we have a functor $\Cov(I,J)\to (\amalg_J)_{U_I/}$, defined in exactly the same way, and, as before, this functor is initial. Since $\Cov(I,J)$ is finite, Corollary \ref{cor:right day} implies that right Day convolution exists on $\Fun(\widetilde{\D}_G,\C)$ and that the tensor product in this monoidal structure is given by the formula \[\bigg(\bigotimes_J^\mathrm{RD}F_j\bigg)(U_I)\simeq \bigoplus_{\Cov(I,J)} \bigotimes_J F_j(U_{S_j}).\]

\subsection{Overlapping tensor product}\label{section:overlapping} We now begin the task of transferring this monoidal structure to the $\infty$-category of constructible cosheaves on the $\Ran_G$. This will require a few additional notions.

\begin{definition}
Let $M$ be a $G$-framed manifold. We say that an $I$-tuple $\{U_i\subseteq M\}_{i\in I}$ of objects of $\Diag_G$ is \emph{disjoint} if $U_i\cap U_j=\varnothing$ for $i\neq j\in I$. 
\end{definition}

Note that, if $\{U_i\subseteq M\}_{i\in I}$ is disjoint, then $U_I\subseteq M$ is an object of $\Ran_G$; thus, we may locate $\Ran_G$ within a larger auxiliary object, which we now define.

\begin{definition}
We define a category $\widetilde\Ran_G$ as follows:
\begin{enumerate}
\item an object of $\widetilde\Ran_G$ is a $G$-framed manifold $M$ and an $I$-tuple $\{U_i\subseteq M\}_{i\in I}$ of objects of $\Diag_G$;
\item a morphism from $\{U_i\subseteq M\}_{i\in I}$ to $\{V_j\subseteq N\}_{j\in J}$ is a $J$-cover $S$ of $I$ such that $\{U_i\subseteq M\}_{i\in S_j}$ is disjoint for each $j\in J$, together with an element of $\prod_J\Hom_{\Ran_G}(U_{S_j}\subseteq M, V_j\subseteq N)$;
\item composition is given by composition of covers and composition in $\Ran_G$. 
\end{enumerate}
\end{definition}

Our constructions fit together into the following commuting diagram: \[\xymatrix{\Ran_G\ar[r]^-\iota\ar[d]_-\epsilon&\widetilde\Ran_G\ar[d]^-{\tilde\epsilon}\\
\D_G\ar[d]\ar[r]&\widetilde\D_G\ar[d]\\
\Fin\ar[r]&\Cov.
}\]

We shall say that a functor from $\widetilde\D_G$ is locally constant if its restriction to $\D_G$ is so, and we shall say that a functor from $\widetilde\Ran_G$ is a constructible cosheaf if its restriction to $\Ran_G$ is so. The argument of Lemma \ref{lem:cosheaf identification} yields the following relationship between these two types of functors:

\begin{lemma}\label{lem:cosheaf identification 2}
Restriction along $\epsilon:\widetilde\Ran_G\to\widetilde\D_G$ induces an equivalence \[\xymatrix{
\Fun^\loc(\widetilde\D_G, \C)\ar[r]^-\sim&\cShv^\cbl(\widetilde\Ran_G,\C).
}\]
\end{lemma}

Since it is clear from the formula for right Day convolution on $\Fun(\widetilde\D_G,\C)$ that the tensor product of locally constant functors is again locally constant, we obtain in this way a symmetric monoidal structure on $\cShv^\cbl(\widetilde\Ran_G,\C)$. Our next result asserts that the $\infty$-category of constructible cosheaves on $\Ran_G$ is a localization of this larger $\infty$-category.

\begin{definition}
We say that a constructible cosheaf $F:\widetilde \Ran_G\to \C$ has \emph{disjoint support} if $F(\{U_i\subseteq M\}_{i\in I})\simeq0$ whenever $\{U_i\subseteq M\}_{i\in I}$ is not disjoint.
\end{definition}

\begin{lemma}
The right Kan extension $\iota_*:\cShv^\cbl(\Ran_G,\C)\to \Fun(\widetilde\Ran_G,\C)$ is fully faithful with essential image the constructible cosheaves with disjoint support. 
\end{lemma}
\begin{proof}
From the definitions, \[(\{U_i\subseteq M\}_{i\in I})_{/\iota}=\begin{cases}
(\{U_i\subseteq M\}_{i\in I})_{/\Ran_G}&\quad \{U_i\subseteq M\}_{i\in I} \text{ is disjoint}\\
\varnothing&\quad \text{else.}
\end{cases}\]
\end{proof}

Next, we show that this localization is compatible with the monoidal structure inherited from right Day convolution on $\widetilde \D_G$.

\begin{lemma}
Let $\varphi:F_1\to F_2$ be a morphism and $F$ an object in $\cShv^\cbl(\widetilde\Ran_G,\C)$. If $\iota^\natural \varphi$ is an equivalence, then so is $\iota^\natural(\varphi\otimes \id_F)$.
\end{lemma}
\begin{proof}
Evidently, if $\{U_i\subseteq M\}_{i\in I}$ is disjoint, then any subset is again disjoint. Thus, we have a commuting diagram $$\xymatrix{
\iota^\natural (F_1\otimes F)(U_I\subseteq M)\ar[d]_-{\iota^\natural (\varphi\otimes F)(U_I\subseteq M)}\ar[r]^-\sim&\displaystyle\bigoplus_{I_1\cup I_2=I}\iota^\natural F_1(U_{I_1}\subseteq M)\otimes \iota^\natural F(U_{I_2}\subseteq M)\ar[d]^-{\bigoplus \iota^\natural\varphi(U_{I_1}\subseteq M)\otimes \iota^\natural\id_{F(U_{I_2}\subseteq M)}}\\
\iota^\natural(F_2\otimes F)(U_I\subseteq M)\ar[r]^-\sim&\displaystyle\bigoplus_{I_1\cup I_2=I}\iota^\natural F_2(U_{I_1}\subseteq M)\otimes \iota^\natural F(U_{I_2}
\subseteq M).
}$$ The righthand vertical arrow is an equivalence by assumption, so the lefthand vertical arrow is also an equivalence by two-out-of-three.
\end{proof}

By \cite[2.2.1.9]{Lurie:HA}, we obtain a second symmetric monoidal structure on $\cShv^\cbl(\Ran_G,\C)$ for which the restriction $\iota^\natural$ is symmetric monoidal. We refer to this monoidal structure as the \emph{overlapping monoidal structure}. Where necessary, the symbol $\cup$ will be used to disambiguate the overlapping tensor product from other tensor products (e.g. $\otimes^\cup$).

The lax monoidal structure on the identity asserted in Proposition \ref{prop:overlapping exists} is obtained after localization from the oplax structure on the restriction $\Fun(\widetilde \D_G,\C)\to \Fun(\D_G,\C)$ arising from the fact that $\D_G\to \widetilde \D_G$ is symmetric monoidal. 

\begin{remark}
In terms of explicit formulas, the components of the lax monoidal structure are given by the projections \[\bigotimes_J^\cup F_j(U_I)\simeq \bigoplus_{\Cov(I,J)}\bigotimes_JF_j(U_{S_j})\to \bigoplus_{\Fin(I,J)}\bigotimes_JF_j(U_{f^{-1}(j)})\simeq\bigotimes_J^\amalg F_j(U_I)\] induced by the inclusion $f\mapsto \{f^{-1}(j)\}_{j\in J}$ of functions into covers. 
\end{remark}

With the following lemma, we complete the proof of Proposition \ref{prop:overlapping exists}:

\begin{lemma}\label{lem:overlapping is pushforward}
There is a natural equivalence $F\otimes^\cup G\simeq \mu_*(F\boxtimes G)$.
\end{lemma}
\begin{proof}
Consider the \emph{noncommutative} diagram:
\[\xymatrix{
(\Ran_G\times_{\Mfld_G}\Ran_G)_\#
\ar[r]\ar[dd]_-\mu\ar[r]^-f&\Ran_G\times_{\Mfld_G}\Ran_G\ar[r]^-{\iota\times\iota}&\widetilde\Ran_G\times_{\Mfld_G}\widetilde\Ran_G\ar[r]^-{\tilde\epsilon\times\tilde\epsilon}&\widetilde\D_G\times\widetilde\D_G\ar[dd]^-\amalg\\\\
\Ran_G\ar[rr]^{\iota}&&\widetilde\Ran_G\ar[r]^-{\tilde\epsilon}&\widetilde\D_G.
}\] Although the two composites do not agree, there is an evident map $\tilde\epsilon\iota \mu\to\amalg(\tilde\epsilon\times\tilde\epsilon)(\iota\times\iota)f$, which induces a natural transformation \[\mu_*(F\boxtimes G)\to F\otimes^\cup G\] after taking right Kan extensions. To see that this map is an equivalence, we note that, as in Lemma \ref{lem:multiplication push formula} and \S\ref{sec:overlapping disks}, each of the undercategories in question receives an initial functor from $\Cov(-,2)$ (note that $\mu^{-1}(U_I\subseteq M)$ is isomorphic to $\Cov(I, 2)$).
\end{proof}

From Lemmas \ref{lem:global sections} and \ref{lem:overlapping is pushforward}, we see that $\pi_!$ lifts to a symmetric monoidal functor between the overlapping and pointwise monoidal structures on $\Fun(\Mfld_G,\C)$. We record the following fact concerning the interaction of this functor with Lie chains, which is immediate from \cite[6.2.6]{FrancisGaitsgory:CKD}:

\begin{lemma}\label{lem:globalization chains}
The diagram $$\xymatrix{
\Alg_\Lie(\cShv^\cbl(\Ran_G,\C)_\cup)\ar[d]_-{C_\cup^\Lie}\ar[r]^-{\pi_!}&\Alg_\Lie(\Fun^\loc(\Mfld_G,\C)_\pt)\ar[d]^-{C^\Lie_\amalg}\\
\Coalg_{\Com}(\cShv^\cbl(\Ran_G,\C)_\cup)\ar[r]^-{\pi_!}&\Coalg_{\Com}(\cShv^\cbl(\Ran_G,\C)_\pt)}$$ commutes.
\end{lemma}

\bibliographystyle{amsalpha}
\bibliography{references}

\end{document}